\documentclass[12pt]{article}
\usepackage{amssymb}
\usepackage{amsmath,amsthm}
\usepackage[utf8]{inputenc}
\usepackage[dvips]{graphicx}
\usepackage{hyperref}
\usepackage{enumerate}
\usepackage{color}
\usepackage{tikz,ifthen}
\usepackage[T1]{fontenc}
\usepackage{float}

\usepackage{lineno}

\newtheorem{remark}{Remark}

\newtheorem{definition}{Definition}

\newtheorem{theorem}[remark]{Theorem}
\newtheorem{observation}[remark]{Observation}
\newtheorem{proposition}[remark]{Proposition}
\newtheorem{corollary}[remark]{Corollary}
\newtheorem{claim}[remark]{Claim}

\newcommand{\smallqed}{{\tiny ($\Box$)}}
\newcommand{\rmv}[1]{}

\def\diam{{\rm diam}}

\begin{document}


\title{On Distance and Strong Metric Dimension of the Modular Product}
\author{Cong X. Kang$^{(1)}$, Aleksander Kelenc$^{(2,3)}$,  Iztok Peterin$^{(2,4)}$, Eunjeong Yi$^{(1)}$\\
\\
$^{(1)}${\small Texas A\&M University at Galveston, Galveston, TX 77553, USA}\\
$^{(2)}$ {\small Faculty of Electrical Engineering and Computer Science}\\
{\small University of Maribor,} {\small Koro\v{s}ka cesta 46, 2000 Maribor,
Slovenia.} \\
$^{(3)}$ {\small Faculty of Natural Sciences and Mathematics}\\
{\small University of Maribor,} {\small Koro\v{s}ka cesta 160, 2000 Maribor,
Slovenia.} \\
$^{(4)}${\small Institute of Mathematics, Physics and Mechanics}\\
{\small Jadranska ulica 19, 1000 Ljubljana, Slovenia.} \\
{\small \texttt{e-mail:} \textit{kangc\@@tamug.edu}; \textit{aleksander.kelenc\@@um.si}; \textit{iztok.peterin\@@um.si}; \textit{yie\@@tamug.edu}} \\
}

\maketitle

\date{}

\begin{abstract}
The \emph{modular product} $G\diamond H$ of graphs $G$ and $H$ is a graph on vertex set $V(G)\times V(H)$. Two vertices $(g,h)$ and $(g',h')$ of $G\diamond H$ are adjacent if $g=g'$ and $hh'\in E(H)$, or $gg'\in E(G)$ and $h=h'$, or $gg'\in E(G)$ and $hh'\in E(H)$, or (for $g\neq g'$ and $h\neq h'$) $gg'\notin E(G)$ and $hh'\notin E(H)$. We derive the distance formula for the modular product and then describe all edges of the strong resolving graph of $G\diamond H$. This is then used to obtain the strong metric dimension of the modular product on several, infinite families of graphs.
\end{abstract}



\section{Introduction}

Taking a product of graphs represents a way of building a bigger object (the product) from smaller objects (the factors) in a systematic way. Conversely, decomposing a graph into its component factors often yields benefits. In particular, it is often possible to determine or estimate a property of the product graph from the (same or some different) properties of the factor graphs. Moreover, such a factorization often facilitates algorithmic explorations.

The study of graph products can be split into two main strands. The first is to check whether a graph $G$ is a product at all; i.e., can we find (smaller) graphs $G_1,\dots,G_k$ such that, for some graph product $*$, $G=G_1*\cdots*G_k$? Moreover, is such a factorization unique?
These questions were answered in the affirmative for several graph products in some graph classes, and polynomial algorithms were presented for such factorizations. The main focus was on four standard graph products: the Cartesian, the strong, the direct, and the lexicographic. For the Cartesian product, the optimal (i.e. linear) algorithm was presented in \cite{ImPe} after two decades of development in the area. A polynomial algorithm can be found for factorization in the strong product \cite{FeSc} and in the direct product \cite{Imri}. A different approach, via the so-called Cartesian skeletons, was considered in \cite{HaIm}. The factorization in the lexicographic product is connected with the graph isomorphism problem. For details on the aforementioned results, please see the exhaustive monograph on graph products \cite{HaIK}.

The other strand of inquiry studies the connections between the product and its factors. This enterprise has been immensely popular in the graph theory community in recent decades. One result of this type is Hedetniemi's conjecture $\chi(G\times H)=\min\{\chi(G),\chi(H)\}$ for the chromatic number of the direct product; the conjecture was disproved recently in \cite{Shit}. Another example of the type is Vizing's conjecture $\gamma(G\Box H)\geq\gamma(G)\gamma(H)$ for the domination number of the Cartesian product; see the most recent survey \cite{BrDoKl} for a discussion of the conjecture, which remains open. Both conjectures have inspired many deep results over the years.

While there are two hundred fifty-six different graph products such that the edge set of each product is determined by the edge sets of its two factors, only twenty of the products are associative; see \cite{Imri1,ImIz}. Moreover, only ten of the twenty are also commutative, with the empty product and the complete product being among them and regarded as trivial. The remaining products are grouped into pairs, where each pair is a product $G*H$ and its complementary product $G\overline{*}H\stackrel{\mathrm{def}}{=}\overline{\overline{G}*\overline{H}}$ (with $\overline{G}$ denoting the complement of $G$). Three pairs are the aforementioned standard products -- namely the Cartesian, the direct, and the strong, together with their complementary products. (Note that the lexicographic product is not commutative.) This leaves in the group the final product, called the \emph{modular product}. One could also call it ``the forgotten (associative and commutative) product" because, after the preliminary work on the modular product \cite{Imri1}, there is, to our knowledge, only one additional publication \cite{shao-solis} on it. The paper \cite{Imri1} investigates the modular product with respect to the unique factorization and some other properties, while the paper \cite{shao-solis} bears some results on $L(2,1)$-labelings.

In the present work, we seek to rehabilitate the modular product back into the spotlight. In the next section, we settle the terminology with emphasis on graph products and the strong metric dimension. Then, in the third section, we derive the distance formula between two vertices of the modular product; this reveals the interesting fact that modular product graphs are often of diameter two, while their factors are not. In the last section, we first characterize all edges belonging to the strong resolving graph of a modular product; as an application, we obtain the strong metric dimension of the modular product for several families of graph.

\section{Preliminaries}

All graphs considered herein are simple and undirected. Let $G$ be a graph. The \emph{complement} $\overline{G}$ of $G$ is a graph with $V(\overline{G})=V(G)$, and two vertices are adjacent in $\overline{G}$ whenever they are not adjacent in $G$. The \emph{distance} $d_{G}(u,v)$ between vertices $u,v\in V(G)$ is the minimum number of edges on any path between $u$ and $v$ in $G$. If there exists no path between $u$ and $v$, then we set $d_G(u,v)=\infty$. A \emph{shortest path}, thus comprising $d_G(u,v)$ edges, between vertices $u$ and $v$ is called a $u,v$-\emph{geodesic}. Sometimes, a shortest walk of odd or even length plays an important role. Hence, we denote by $d_G^o(u,v)$ the minimum odd number of edges on any $u,v$-walk and set $d_G^o(u,v)=\infty$ if such a walk does not exist. Similarly, $d_G^e(u,v)$ denotes the minimum even number of edges on any $u,v$-walk, and we have $d_G^e(u,v)=\infty$ if such a walk does not exist. The maximum distance between any two vertices of $G$ is the \emph{diameter} of $G$, and it is denoted by ${\rm diam}(G)$.

As usual, we denote the open neighborhood and the closed neighborhood of a vertex $v$ of $G$ by $N_G(v)$ and $N_G[v]$, respectively. If $N_G[v]=V(G)$, then $v$ is called a \emph{universal vertex}. If $N_G(v)=\emptyset$, then $v$ is called an \emph{isolated vertex}. We use $\overline{N}_G[v]$ for the complement of $N_G[v]$ in $V(G)$. Vertices $u$ and $v$ of $G$ are \emph{twins} if $N_G[u]=N_G[v]$. Notice that the twin relation defines an equivalence relation on $V(G)$. An edge $uv$ between twins is called a \emph{twin edge}, and by $TW(G)$ we denote the set of all twin edges of $G$. If $N_G(u)=N_G(v)$, then $u$ and $v$ are called \emph{false twins}. The false-twin relation also defines an equivalent relation on $V(G)$. We denote the disjoint union of two graphs $G$ and $H$ by $G\cup H$, and we further ease the notation with $2G=G\cup G$.

A set $D$ is a \emph{dominating set} of $G$ if $\bigcup_{u\in D}N[u]=V(G)$. The minimum cardinality of a dominating set is the  \emph{domination number} of $G$ and it is denoted by $\gamma(G)$. A dominating set of cardinality $\gamma (G)$ is called a $\gamma(G)$-set. If a $\gamma(G)$-set $\{g,g'\}$ satisfies $d_G(g,g')=3$, then we have $N_G[g]\cap N_G[g']=\emptyset$, in addition to having $N_G[g]\cup N_G[g']=V(G)$. In this case, $\{N_G[g],N_G[g']\}$ partitions $V(G)$ and has been referred to as a \emph{perfect code} of $G$, and $G$ is a so-called \emph{efficient closed dominated graph}; see \cite{KPY} and the references therein for more information on these graphs. We will call a pair $\{g,g'\}$ a $\gamma_G$-pair when it is a perfect code of $G$. Notice that
\begin{equation}\label{gammapair}
N_G[g]=\overline{N}_G[g'] \ (\text{equivalently } N_G[g']=\overline{N}_G[g]) \text{ for a } \gamma_G\text{-pair } \{g,g'\}.
\end{equation}

A set $B\subseteq V(G)$ is a \emph{vertex cover set} of $G$ if every edge of $G$ is incident with at least one vertex of $B$. The minimum cardinality of a vertex cover set $B$ of $G$ is the \emph{vertex cover number} of $G$ and is denoted by $\beta(G)$. Every vertex cover set of $G$ of cardinality $\beta(G)$ is called a $\beta(G)$-set. A set $A\subseteq V(G)$ is an \emph{independent set} of $G$ if no pair of vertices of $A$ are adjacent in $G$. A consequence of the famous Gallai formula is the well-known fact that the complement of an independent set is a vertex cover, and vice versa. We will use $[k]$ for $\{1,\dots,k\}$.

\subsection{Products of Graphs}

The different products defined between $G$ and $H$ herein all have vertex set $V(G)\times V(H)$, with different edge sets. Two vertices $(g,h)$ and $(g',h')$ are adjacent in the \emph{modular product} $G\diamond H$ if $g=g'$ and $hh\in E(H)$, or $gg'\in E(G)$ and $h=h'$, or $gg'\in E(G)$ and $hh'\in E(H)$, or (for $g\neq g'$ and $h\neq h'$) $gg'\notin E(G)$ and $hh'\notin E(H)$. The first two conditions define the edges of the Cartesian product $G\Box H$ and are therefore called the \emph{Cartesian edges}. The edges arising from the third condition are called \emph{direct}, as they are the edges of the \emph{direct product} $G\times H$. The edges arising from the final condition are called \emph{co-direct} edges, because they are the edges of the direct product of $\overline{G}\times\overline{H}$. So, we have
$$E(G\diamond H)=E(G\Box H)\cup E(G\times H)\cup E(\overline{G}\times\overline{H}).$$
Recall that the edge set of the \emph{strong product} $G\boxtimes H$ consists of the Cartesian edges and the direct edges between $G$ and $H$. The \emph{lexicographic product} $G\circ H$ is defined by the condition of $(g,h)(g',h')$ being an edge in $G\circ H$ if $gg'\in E(G)$ or $g=g'$ and $hh'\in E(H)$. Finally, the direct-co-direct product $G\ast H$ has $E(G\ast H)=E(G\times H)\cup E(\overline{G}\times\overline{H})$. Notice that, among the aforementioned products, only the lexicographic product is not commutative due to the asymmetrical definition of its edge set and only direct-co-direst product is not associative. Different products can be isomorphic for some factors. For us, the following simple connection is important:
\begin{equation}\label{complete}
G\boxtimes K_t\cong G\circ K_t\cong G\diamond K_t.
\end{equation}
Also, $K_s\diamond K_t\cong K_{st}$, $\overline{K}_s\diamond \overline{K}_t\cong K_s\times K_t$ and $(K_p\cup K_r)\diamond (K_s\cup K_t)\cong K_{ps+rt}\cup K_{pt+rs}$ follow easily from the definition of the modular product.

The following result about the connectivity of $G\diamond H$ is from \cite{Imri1} (Theorem 3.1), see also exercise 4.20 on p. 46 in \cite{HaIK}.

\begin{theorem}\label{conn}
The modular product $G\diamond H$ is disconnected if and only if one factor is complete and the other is disconnected, or both factors are disjoint union of two complete graphs.
\end{theorem}

The closed neighborhood of a vertex in the modular product, as seen immediately from the definition, is given by
\begin{equation}\label{modular_nbrs}
N_{G\diamond H}[(g,h)]=(N_G[g]\times N_H[h])\cup (\overline{N}_G[g]\times \overline{N}_H[h]).
\end{equation}

The idea of \emph{twins} will be important for the rest of this paper; therefore, we describe them in light of the modular product. But, we will first dispatch any concern regarding the so-called ``false twins", vis-à-vis modular product. Recall that distinct vertices $x$ and $y$ of a graph $X$ are said to be \emph{false twins} if $N_X(x)=N_X(y)$.

\begin{proposition}
Let both $G$ and $H$ each have order at least two. Then, there are no distinct false twins in $G\diamond H$.
\end{proposition}

\begin{proof}
Since adjacent vertices cannot have the same open neighborhood, we may assume that $(g,h),(g',h')\in V(G\diamond H)$ satisfy $d_{G \diamond H}((g,h), (g',h'))\geq 2$. So, either $g'\in N_G[g]$ and $d_H(h,h')\ge 2$, or $d_G(g, g')\ge 2$ and $h'\in N_H[h]$. Without loss of generality (due to symmetry), let $g'\in N_G[g]$ and $d_H(h, h')\ge 2$.

First, suppose $g=g'$. If there exists $x\in N_G(g)$, then $(x,h) \in N_{G\diamond H} ((g,h))$ but $(x,h) \not\in N_{G \diamond H}((g', h'))$, since $g'x\in E(G)$ but $hh'\not\in E(H)$. Otherwise, $N_G[g]=\{g\}$. Then, any $x'\in V(G)-\{g\}$ satisfies $(x',h') \in N_{G \diamond H} ((g,h))$ but $(x', h') \not\in N_{G \diamond H} ((g', h'))$, since $x'g'\not\in E(G)$.

Second, suppose $g' \in N_G(g)$. Then $(g', h) \in N_{G \diamond H}((g,h))$ and $(g', h) \not\in N_{G \diamond H} ((g', h'))$ since $hh'\not\in E(H)$. In each case, $(g,h)$ has a neighbor who is not a neighbor of $(g',h')$. Therefore, there are no false twins in $G\diamond H$. ~\hfill
\end{proof}

\begin{theorem}\label{twins1}
Let $G$ and $H$ be graphs. Vertices $(g,h)$ and $(g',h')$ are twins of $G\diamond H$ if and only if
\begin{description}
\item[$(i)$]$g$ and $g'$ are twins of $G$ and $h$ and $h'$ are twins of $H$, or
\vspace*{-0.1in}
\item[$(ii)$]$\{g,g'\}$ is a $\gamma_G$-pair and $\{h,h'\}$ is a $\gamma_H$-pair.
\end{description}
\end{theorem}

\begin{proof}
$(\Rightarrow)$ Let $(g,h)$ and $(g',h')$ be twins of $G\diamond H$. By~(\ref{modular_nbrs}), we have
\begin{equation}\label{equality}
\begin{array}{ccc}
N_{G\diamond H}[(g,h)]&=&(N_G[g]\times N_H[h])\cup (\overline{N}_G[g]\times \overline{N}_H[h])=\\
&=&(N_G[g']\times N_H[h'])\cup (\overline{N}_G[g']\times \overline{N}_H[h'])=N_{G\diamond H}[(g',h')].
\end{array}
\end{equation}
We consider two cases. Suppose first that $(g,h)$ and $(g',h')$ are adjacent by a Cartesian or a direct edge. In this case $g,g'\in N_G[g]\cap N_G[g']$ and $h,h'\in N_H[h]\cap N_H[h']$. If $N_G[g]\neq N_G[g']$, say $g_0\in N_G[g]-N_G[g']$, then $(g_0,h)\rmv{(g_0,h')}\in N_{G\diamond H}[(g,h)]$ but $(g_0,h)\rmv{(g_0,h')}\notin N_{G\diamond H}[(g',h')]$, a contradiction. Hence, $N_G[g]=N_G[g']$ and $g$ and $g'$ are twins of $G$. By a symmetric argument, we obtain $N_H[h]=N_H[h']$ and $h$ and $h'$ are twins of $H$. So $(i)$ follows.

Assume secondly that $(g,h)$ and $(g',h')$ are adjacent by a co-direct edge. Now $g\in \overline{N}_G[g']$, $g'\in \overline{N}_G[g]$, $h\in \overline{N}_H[h']$ and $h'\in \overline{N}_H[h]$. Suppose that $N_G[g]\neq \overline{N}_G[g']$. If $g_0\in N_G[g]-\overline{N}_G[g']$, then $g_0$ is a common neighbor of $g$ and $g'$ and $(g_0,h)\in N_{G\diamond H}[(g,h)]$ but $(g_0,h)\notin N_{G\diamond H}[(g',h')]$ because $h$ and $h'$ are different and not adjacent, a contradiction. So, $N_G[g]-\overline{N}_G[g']=\emptyset$ and we may assume that $g_1\in \overline{N}_G[g']-N_G[g]$. This means that $g_1$ is not adjacent to $g'$ and not adjacent to $g$. Hence, $(g_1,h)\in N_{G\diamond H}[(g',h')]$ but $(g_1,h)\notin N_{G\diamond H}[(g,h)]$ because $h$ and $h'$ are different and not adjacent, a contradiction. Therefore, $N_G[g]=\overline{N}_G[g']$ (equivalently, $\overline{N}_G[g]=N_G[g']$), which means that $\{g,g'\}$ is a $\gamma_G$-pair. By a symmetric argument we obtain that $\{h,h'\}$ is a $\gamma_H$-pair and $(ii)$ follows.

$(\Leftarrow)$ If $g$ and $g'$ are twins of $G$ and $h$ and $h'$ are twins of $H$, then $N_G[g]=N_G[g']$, $\overline{N}_G[g]=\overline{N}_G[g']$, $N_H[h]=N_H[h']$, $\overline{N}_H[h]=\overline{N}_H[h']$. Clearly (\ref{equality}) holds and $(g,h)$ and $(g',h')$ are twins. Let now $\{g,g'\}$ be a $\gamma_G$-pair and $\{h,h'\}$ be a $\gamma_H$-pair. By (\ref{gammapair}) we have
\begin{eqnarray*}
N_{G\diamond H}((g,h))&=&(N_G[g]\times N_H[h])\cup(\overline{N}_G[g]\times \overline{N}_H[h])=\\
&=&(\overline{N}_G[g']\times \overline{N}_H[h'])\cup(N_G[g']\times N_H[h'])=N_{G\diamond H}((g',h'))
\end{eqnarray*}
and, $(g,h)$ and $(g',h')$ are twins again.
\end{proof}

\subsection{Strong metric dimension}

A vertex $z\in V(G)$ \emph{strongly resolves} two different vertices $x,y\in V(G)$ if $x$ belongs to a $y,z$-geodesic, or $y$ belongs to an $x,z$-geodesic. This condition can also be presented with distances as $d_G(y,z)=d_G(y,x)+d_G(x,z)$ or $d_G(x,z)=d_G(x,y)+d_G(y,z)$. If vertices $u$ and $v$ are not resolved by any third vertex $z$, then $u$ and $v$ are called \emph{mutually maximally distant} or MMD for short. A \emph{strong metric generator} in a connected graph $G$ is a set $S\subseteq V(G)$ such that every two vertices of $G$ are strongly resolved by a vertex of $S$. By ${\rm dim}_s(G)$ we denote the minimum cardinality of a strong metric generator for $G$ and we call it the \emph{strong metric dimension} of $G$. A strong metric generator for $G$ of cardinality ${\rm dim}_s(G)$ is called a \emph{strong metric basis} of $G$. Given any ordered set $S=\{w_1,\ldots, w_k\}\subseteq V(G)$ and a vertex $v\in V(G)$, the \emph{code vector} $\rm code_S(v)$ of $v$ with respect to $S$ is the $k$-vector $(d_G(v,w_1), \ldots, d_G(v,w_k))$. 

Strong metric dimension was introduced by Seb\"{o} and Tannier in \cite{seb}. A strong incentive for studying the strong metric dimension, as it was already pointed out in \cite{seb}, is that the set of all code vectors of a graph $G$ relative to a strong metric generator uniquely determines the graph -- in contrast to the situation with its more famous progenitor, the notion of \emph{metric dimension}; see \cite{KY} for a proof of this unique determination. The most influential paper on this topic is from Oellermann and Peters-Fransen \cite{Oellermann} and we describe the approach developed there. The idea is to transform a graph $G$ into a graph $G_{SR}$, called the \emph{strong resolving graph} of $G$, and the following result connects the strong metric dimension of $G$ with the vertex cover problem of $G_{SR}$. 

\begin{theorem}{\em \cite{Oellermann}}\label{lem_oellerman}\label{good}
For any connected graph $G$,
${\rm dim}_s(G) = \beta(G_{SR}).$
\end{theorem}

Let us define $G_{SR}$ in a manner slightly different from \cite{Oellermann}. A $u,v$-geodesic $P$ is \emph{maximal} if $P$ is not contained in any geodesic different from $P$. \rmv{(Minimum cardinality of a vertex set that traverses all the maximal geodesics in $G$ were recently introduced independently in \cite{PeSe} under the name maximal shortest paths cover number and in \cite{MaBK}.)} Vertices of $G_{SR}$ are all the end-vertices of all maximal geodesics in $G$. Moreover, $uv\in E(G_{SR})$ if there exists a maximal $u,v$-geodesic in $G$. It is sometimes more convenient to show $uv\notin E(G_{SR})$: this means that there exists an $x\in N_G(u)$ such that $u$ is on an $x,v$-geodesic, or there exists $y\in N_G(v)$ such that $v$ is on a $y,u$-geodesic. In terms of distances, we have $d_G(x,v) =d_G(u,v)+1$ for some $x$, or $d_G(y,u) =d_G(v,u)+1$ for some $y$, when $uv\notin E(G_{SR})$. Notice that in \cite{Oellermann}, the strong resolving graph was defined for all vertices of $G$, not only for the end-vertices of maximal geodesics. 

Theorem \ref{lem_oellerman} clearly suggests, as already shown in \cite{Oellermann}, that the problem of computing ${\rm dim}_s(G)$ is NP-hard. Therefore, any result that yields a polynomial algorithm for computing the strong metric dimension for special classes of graphs is welcome; the rationale applies also to the reduction of the problem to simpler graphs. This comes in handy for product graphs, where the study of a given parameter can often be reduced to the study of the same or some different parameter for the factors of the product. Strong metric dimension was studied in \cite{Oellermann,str-dim-cart-dir} for the Cartesian product, in \cite{strong-lex,Strong-strong,Kuziak-Erratum} for the strong product, in \cite{strong-lex} for the lexicographic product, and in \cite{str-dim-cart-dir,KuPY} for the direct product. In this paper, we continue the study of strong metric dimension on products -- this time, on the modular product. The following observation is straightforward.

\begin{observation}\label{twins2}
Let $G$ be a graph. Vertices $u$ and $v$ are different twins if and only if $uv\in E(G)\cap E(G_{SR})$.
\end{observation}

\section{Distance}

The most fundamental metric information on a graph is the distance between two of its vertices. In a graph product, this usually means describing the distance between two vertices $(g,h)$ and $(g',h')$ in terms of the distances $d_G(g,g')$ and $d_H(h,h')$. This can be seen in distance formulas for different products. We infer from \cite{HaIK} that
\begin{equation}\label{cdist}
d_{G\Box H}((g,h),(g',h'))=d_G(g,g')+d_H(h,h')
\end{equation}%
for the Cartesian product. For the strong product we have
\begin{equation}\label{sdist}
d_{G\boxtimes H}((g,h),(g',h'))=\max\{d_G(g,g'),d_H(h,h')\}
\end{equation}%
and in the case of the lexicographic product we have
\begin{equation}\label{ldist}
d_{G\circ H}((g,h),(g',h'))=\left\{
\begin{array}{ccc}
d_G(g,g') & : & g\neq g'\\
\min\{d_H(h,h'),2\} & : & g=g'
\end{array}%
\right. .
\end{equation}%
All these formulas are a part of folklore now, while the formula for the direct product is not so well known. It is
\begin{equation}\label{ddist}
d_{G\times H}((g,h),(g',h'))=\min\{\max\{d^e_G(g,g'),d^e_H(h,h')\},\max\{d^o_G(g,g'),d^o_H(h,h')\}\}
\end{equation}%
and was first proven in \cite{Kim}. A distance formula for direct-co-direct product was recently presented in \cite{KePe} for the connected factors and in \cite{KePe1} for the case of (at least one) disconnected factors. As we will see, formula (\ref{ddist}) mostly does not impact the modular product, due to the abundance of edges. However, there are some exceptions, and we will start with them.

As mentioned $K_p\diamond K_r\cong K_{pr}$ and the distance is trivial in this case. Next case is $\overline{K}_p\diamond \overline{K}_r\cong K_p\times K_r$. With the use of (\ref{ddist}) we obtain
\begin{equation}\label{ldist1}
d_{\overline{K}_p\diamond \overline{K}_r}((g,h),(g',h'))=\left\{
\begin{array}{ccc}
0 & : & g=g'\wedge h=h'\\
1 & : & g\neq g'\wedge h\neq h'\\
2 & : & g=g'\veebar h=h'
\end{array}%
\right.
\end{equation}%
when $p,r\geq 3$. There is a small change when $r=2$ and $p\geq 3$. Again we use (\ref{ddist}) and get
\begin{equation}\label{ldist2}
d_{\overline{K}_p\diamond \overline{K}_2}((g,h),(g',h'))=\left\{
\begin{array}{ccc}
0 & : & g=g'\wedge h=h'\\
1 & : & g\neq g'\wedge h\neq h'\\
2 & : & g\neq g'\wedge h=h'\\
3 & : & g=g'\wedge h\neq h'
\end{array}%
\right.
\end{equation}%
Finally, $\overline{K}_2\diamond \overline{K}_2\cong K_2\times K_2\cong 2K_2$ and everything is clear here.

Next special case is when one factor, say $H$, is isomorphic to a complete graph and $G$ is not. In such a case we have $G\diamond K_r\cong G\boxtimes K_r\cong G\circ K_r$, and by (\ref{sdist}) or (\ref{ldist}) we have
$$d_{G\diamond K_r}((g,h),(g',h'))=\max\{d_G(g,g'),d_{K_r}(h,h')\}=\left\{
\begin{array}{ccc}
d_G(g,g') & : & g\neq g'\\
d_{K_r}(h,h') & : & g=g'
\end{array}
\right. .$$

Notice that this case covers one of the possible cases for disconnected $G\diamond H$, namely when $G$ is disconnected and $H$ is complete; see Theorem \ref{conn}.

Next, we consider a general situation where we avoid complete graphs. Then, the distance is at most three, with one exception where each of the two factors is comprised of two cliques.	
	
\begin{theorem}\label{leq3}
Let neither $G$ nor $H$ be a complete graph. Then, either $d_{G\diamond H}((g,h),(g',h'))\leq 3$, or $d_{G\diamond H}((g,h),(g',h'))=\infty$ where $G$ and $H$ are both a disjoint union of two cliques.
\end{theorem}
	
\begin{proof}
Let $(g,h),(g',h')\in V(G\diamond H)$ be given. If $g'\notin N_G[g]$ and $h'\notin N_H[h]$, then $(g',h')$ is adjacent to and thus of distance one from $(g,h)$ in $G\diamond H$ by a co-direct edge. So, we may hereinafter assume, by symmetry of the modular product, that $g'\in N_G[g]$.

If $d_H(h,h')=0$, then the distance between $(g,h)$ and $(g',h')$ is either zero when $g=g'$ or one when $g\neq g'$. If $d_H(h,h')=1$, then $(g,h)$ and $(g',h')$ are adjacent by a Cartesian edge when $g=g'$ or by a direct edge when $g\neq g'$.

Now, let $d_H(h,h')=2$ and $hh_1h'$ denote a geodesic in $H$; note that $(g,h)$ and $(g',h')$ are not adjacent in this case. We have a path $(g,h)(g,h_1)(g',h')$ in $G\diamond H$ and therefore $d_{G\diamond H}((g,h),(g',h'))=2$. Next, let $d_H(h,h')=3$ and $hh_1h_2h'$ denote a geodesic in $H$. We have a path $(g,h)(g,h_1)(g,h_2)(g',h')$ in $G\diamond H$, and thus $d_{G\diamond H}((g,h),(g',h'))\leq 3$.

Finally, let $d_H(h,h')>3$. Let $g_1$ and $g_2$ denote two distinct and nonadjacent vertices in $G$. First, suppose that both $g$ and $g'$ are universal vertices in $G$.  We thus find a path $(g,h)(g_1,h)(g_2,h')(g',h')$ of length three in $G\diamond H$. If $g$ is a universal vertex while $g'$ is not, then $(g,h)(g'',h)(g',h')$ is a path of length two in $G\diamond H$ for every $g''\notin N_G[g']$. The case where $g'$ is a universal vertex, while $g$ is not, is symmetrically argued.

Now, with $d_H(h,h')>3$, assume also that neither $g$ nor $g'$ is universal in $G$. First, suppose there exists an independent set $\{h,h',h''\}$ in $H$. Then, there is a path $(g,h)(g'',h'')(g,h')(g',h')$ in $G\diamond H$ of length three when $g\neq g'$, or of length two when $g=g'$, for any $g''\notin N_G[g]$.

Finally, with $d_H(h,h')>3$ and that neither $g$ nor $g'$ is universal in $G$, let $\{h,h'\}$ be a maximal independent set of $H$; i.e., $\{h,h',x\}$ is not independent for any $x\in V(H)-\{h,h'\}$. This implies that $h$ and $h'$ belong to two different connected components $C$ and $C'$, respectively, of $H$. Notice that $C$ and $C'$ are the only connected components of $H$, by the maximality (with respect to independence) of the set $\{h,h'\}$. Moreover, note that $N_H[h]=V(C)$ and $N_H[h']=V(C')$, also by maximality assumption on the set $\{h,h'\}$.

Let $g''\notin N_G[g]$; if there exist distinct and nonadjacent vertices $h_1,h_2\in V(C')$ (thus $h'\notin \{h_1,h_2\}$), then we find a path $(g,h)(g'',h_1)(g,h_2)(g',h')$ in $G\diamond H$ of length three. We likewise find a path of length at most three in $G\diamond H$, when there exist two nonadjacent vertices in $C$.

Hence, we may assume that $C$ and $C'$ are cliques and put $H=C\cup C'\cong K_p\cup K_r$. If $g$ and $g'$ are not twins in $G$, then there exists a geodesic $gg'g''$ or a geodesic $g'gg''$ in $G$. We have thus, up to a transposition of the symbols $g$ and $g'$, a path $(g,h)(g'',h')(g',h')$ of length two in $G\diamond H$. If $N_G[g]=N_G[g']$ and the connected component of $G$ containing $g$ and $g'$ is not a clique, then there must exist a geodesic $gg_1g_2$ in $G$, such that $(g,h)(g_2,h')(g_1,h')(g',h')$ is a path of length three in $G\diamond H$.

The only remaining case is that $g$ and $g'$ reside in a clique $D$ of $G$, with $G$ being disconnected. If there exist two nonadjacent vertices $g_1,g_2\in V(G)-V(D)$, then we have a path $(g,h)(g_1,h')(g_2,h)(g',h')$ of length three in $G\diamond H$. Hence, we may assume that also $V(G)-V(D)$ induces a clique and $G\cong K_s\cup K_t$. By Theorem \ref{conn}, $G\diamond H$ is disconnected. Moreover, it is not hard to see that $(K_s\cup K_t)\diamond (K_p\cup K_r)\cong K_{sp+tr}\cup K_{sr+tp}$, with vertices $(g,h)$ and $(g',h')$ belonging to different connected components of the disjoint union; thus, we have $d_{G\diamond H}((g,h),(g',h'))=\infty$.
\end{proof}

Towards a characterization of two vertices at each distance in the modular product, in view of Theorem~\ref{leq3} and the adjacency characterization~(\ref{modular_nbrs}), we will characterize two vertices at distance three.

\begin{theorem}\label{dist3}
Suppose neither $G$ nor $H$ is a complete graph, and suppose not both $G$ and $H$ equal the disjoint union of two cliques.
Then, $d_{G\diamond H}((g,h),(g',h'))=3$ if and only if at least one of the following two (symmetric) conditions holds
\begin{equation}\label{dim3eq1}
N_G[g]=N_G[g'] \wedge d_H(h,h') \geq 3 \wedge  \left( N_G[g]=V(G) \vee \{ h,h' \} \text{ is a } \gamma_H\text{-pair}\right)
\end{equation}
\begin{equation}\label{dim3eq2}
N_H[h]=N_H[h'] \wedge d_G(g,g') \geq 3 \wedge  \left( N_H[h]=V(H) \vee \{ g,g' \} \text{ is a } \gamma_G\text{-pair}\right)
\end{equation}
\end{theorem}

\begin{proof}
$(\Rightarrow)$ Let $d_{G\diamond H}((g,h),(g',h'))=3$.
Suppose that $g' \in N_G[g] \wedge h' \in N_H[h]$ or $g' \notin N_G[g] \wedge h' \notin N_H[h]$, then $d_{G\diamond H}((g,h),(g',h'))\leq 1$, a contradiction. Therefore, either $g' \in N_G[g] \wedge d_H(h,h')\geq 2$ or $h' \in N_H[h] \wedge d_G(g,g')\geq 2$ holds. If $g' \in N_G[g] \wedge d_H(h,h') = 2$ (resp., $h' \in N_H[h] \wedge d_G(g,g') = 2$), then there exists a path $(g,h)(g,h_1)(g',h')$ (resp., a path $(g,h)(g_1,h)(g',h')$) for a vertex $h_1\in N_H[h] \cap N_H[h']$ (resp., for a vertex $g_1 \in N_G[g] \cap N_G[g']$). Thus, we have $d_{G\diamond H}((g,h),(g',h'))= 2$ in either case, a contradiction. \\

Thus, the hypothesis implies $g' \in N_G[g] \wedge d_H(h,h')\geq 3$ or $h' \in N_H[h] \wedge d_G(g,g')\geq 3$. We will argue the case when $g' \in N_G[g] \wedge d_H(h,h')\geq 3$, since the other case is symmetric. Suppose now that $N_G[g]\neq N_G[g']$, and assume WLOG that there exists $g_1 \in N_G[g]-N_G[g']$. Then, there is a path $(g,h)(g_1,h)(g',h')$ in $G\diamond H$, a contradiction. Thus, $N_G[g]=N_G[g']$ and $d_H(h,h')\geq 3$ holds. \\

Suppose now, for contradiction, that $N_G[g]\neq V(G)$ and $\{ h,h' \}$ is not a $\gamma_H$\text-pair. It follows that, given $N_H[h]\cap N_H[h']=\emptyset$ from $d_H(h,h')\geq 3$, there exist a vertex $g_1\in V(G)-N_G[g]=V(G)-N_G[g']$ and a vertex $h_1\in V(H)- (N_H[h]\cup N_H[h'])$. This implies the path $(g,h)(g_1,h_1)(g',h')$ in $G\diamond H$, a contradiction. It follows that either $g$ is a universal vertex or $\{h,h'\}$ is a $\gamma_H$-pair. \\

$(\Leftarrow)$ Suppose now that either (\ref{dim3eq1}) or (\ref{dim3eq2}) holds. We will argue the case (\ref{dim3eq1}); the symmetrical case (\ref{dim3eq2}) can be argued symmetrically. By Theorem \ref{leq3}, $d_{G\diamond H}((g,h),(g',h'))\leq 3$ holds. Noting that the condition $d_H(h,h')\geq 3$ implies $(g,h)\neq (g',h')$, we will show that $d_{G\diamond H}((g,h),(g',h'))\notin \{1,2\}$.\\

First, notice that $d_{G\diamond H}((g,h),(g',h'))=1$ and $N_G[g]=N_G[g']$ imply $d_H(h,h')\leq 1$, contradicting the condition $d_H(h,h')\geq 3$. \\

So, suppose $d_{G\diamond H}((g,h),(g',h'))=2$, meaning that there is a vertex $(g_1,h_1)$ in the intersection of $N_{G\diamond H}[(g,h)]$ and $N_{G\diamond H}[(g',h')]$. Recall the neighborhood factorization formula (\ref{modular_nbrs}), viz. $N_{G\diamond H}[(g,h)]=(N_G[g]\times N_H[h])\cup (\overline{N}_G[g]\times \overline{N}_H[h])$, and call a neighbor $(g_1,h_1)$ either "strong" or "co-direct" depending on whether it belongs to the left or the right member of the disjoint union. By the condition $N_G[g]=N_G[g']$, any common neighbor is either strong to both vertices $(g,h)$ and $(g',h')$, or co-direct to both vertices. Since $d_H(h,h')\geq 3$ (equivalently $N_H[h]\cap N_H[h']=\emptyset$), $(g,h)$ and $(g',h')$ do not have common strong neighbors. So, let $(g_1,h_1)$ be a common co-direct neighbor. Then $g_1\in \overline{N}_G[g]$ and $h_1\in \overline{N}_H[h]\cup \overline{N}_H[h']$, contradicting the condition that either $g$ is universal or $\{h,h'\} \text{ is a }\gamma_H\text{-pair}$.
\end{proof}

Graphs of diameter two constitute a diverse and especially interesting class of graphs; for example, this class includes the celebrated strongly regular graphs. Indeed, in a precise sense, almost all graphs are of diameter two (see, for example, section 7.1 of~\cite{random_graph}). The class also plays an interesting role in the studies of metric dimension and its variants; see~\cite{KPY1, KPY2} for examples. Thus, we distill from the forgoing theorem the following result.

\begin{corollary}\label{diam2}
Suppose neither $G$ nor $H$ is a complete graph, and suppose not both $G$ and $H$ equal the disjoint union of two cliques.
Then, ${\rm diam}(G\diamond H)=2$ if and only if one of the following conditions holds: (1) both factors have diameter two; (2) the event where there is a universal vertex in one factor and there is a $\gamma$-pair in the other factor does not occur; (3) one factor has diameter two and no universal vertex, while the other factor has no $\gamma$-pair.
\end{corollary}

With this we can write the distance formula for modular product under the assumption that factors are not complete graphs and at most one is isomorphic to $K_s\cup K_t$. We have
\begin{equation}\label{moddist}
d_{G\diamond H}((g,h),(g',h'))=\left\{
\begin{array}{ccc}
0 & : & g=g'\wedge h=h'\\
1 & : & (g,h)(g',h')\in E(G\diamond H)\\
2 & : & \text{ otherwise}\\
3 & : & (g,h), (g',h') \text{ fulfills }(\ref{dim3eq1}) \text{ or } (\ref{dim3eq2})
\end{array}%
\right.
\end{equation}%
Notice that (\ref{ldist2}) is of this type, where $V(\overline{K}_2)$ is the desired dominating set to fulfill condition (\ref{dim3eq1}). Also (\ref{ldist1}) is a version of (\ref{moddist}), where distance three is not possible as (\ref{dim3eq1}) \text{ or } (\ref{dim3eq2}) are not fulfilled. Formula (\ref{moddist}) is schematically presented in Figure \ref{scheme}. On the left we have the case when $g$ and $h$ are not universal vertices of $G$ and $H$, respectively, while in the right part $g$ is a universal vertex of $G$. In this case $U$ denotes the other universal vertices of $G$ and $R$ represents all the non-universal vertices of $G$. Clearly, $U$ can be empty, but $R$ is nonempty under assumption that $G$ is not complete.


\begin{figure}[ht!]
\begin{center}
\begin{tikzpicture}[scale=.8,style=thick,x=1cm,y=1cm]
\def\vr{2.5pt} 

\draw [fill=white] (0,0) rectangle (8,7.6);
\draw [fill=white] (0.2,1) rectangle (0.8,2.4);
\draw [fill=white] (0.2,2.6) rectangle (0.8,3.9);
\draw [fill=white] (0.2,4.1) rectangle (0.8,5.4);
\draw [fill=white] (0.2,5.6) rectangle (0.8,7.4);

\draw [fill=white] (1,0.2) rectangle (2.4,0.8);
\draw [fill=white] (1,1) rectangle (2.4,2.4);
\draw [fill=white] (1,2.6) rectangle (2.4,3.9);
\draw [fill=white] (1,4.1) rectangle (2.4,5.4);
\draw [fill=white] (1,5.6) rectangle (2.4,7.4);

\draw [fill=white] (2.6,0.2) rectangle (3.9,0.8);
\draw [fill=white] (4.1,0.2) rectangle (5.9,0.8);
\draw [fill=white] (6.1,0.2) rectangle (7.8,0.8);
\draw [fill=white] (2.6,1) rectangle (3.9,2.4);
\draw [fill=white] (4.1,1) rectangle (5.9,2.4);
\draw [fill=white] (6.1,1) rectangle (7.8,2.4);

\draw [fill=white] (2.6,2.6) rectangle (7.8,7.4);

\draw (4,-0.8) ellipse (4cm and 0.6cm);
\draw (1.7,-0.8) ellipse (0.7cm and 0.4cm);
\draw (3.2,-0.8) ellipse (0.7cm and 0.5cm);
\draw (5,-0.8) ellipse (0.9cm and 0.5cm);
\draw (7,-0.8) ellipse (0.9cm and 0.4cm);

\draw (-0.8,3.8) ellipse (0.6cm and 3.8cm);
\draw (-0.8,1.7) ellipse (0.4cm and 0.7cm);
\draw (-0.8,3.2) ellipse (0.5cm and 0.7cm);
\draw (-0.8,4.7) ellipse (0.5cm and 0.7cm);
\draw (-0.8,6.5) ellipse (0.4cm and 0.9cm);

\path (0.5,0.5) coordinate (a);

\draw (a) [fill=white] circle (\vr);
\draw (0.5,-0.8) [fill=white] circle (\vr);
\draw (-0.8,0.5) [fill=white] circle (\vr);

\draw (0.5,-1.4) node {$g$};
\draw (-1.3,0.5) node {$h$};
\draw (4,-1.8) node {$G$};
\draw (-1.7,3.8) node {$H$};
\draw (3.5,7.9) node {$G\diamond H$};

\draw (0.5,1.7) node {1};
\draw (1.7,0.5) node {1};
\draw (1.7,1.7) node {1};
\draw (5.2,5) node {1};
\draw (0.5,3.2) node {2};
\draw (0.5,6.5) node {2};
\draw (3.2,0.5) node {2};
\draw (7,0.5) node {2};
\draw (1.7,3.2) node {2};
\draw (1.7,6.5) node {2};
\draw (3.2,1.7) node {2};
\draw (7,1.7) node {2};
\draw (0.5,4.7) node {2/3};
\draw (1.7,4.7) node {2/3};
\draw (5,0.5) node {2/3};
\draw (5,1.7) node {2/3};
\draw (1.7,-0.8) node {1};
\draw (3.2,-0.8) node {2};
\draw (5,-0.8) node {3};
\draw (7,-0.8) node {$\geq 4$};
\draw (-0.8,1.7) node {1};
\draw (-0.8,3.2) node {2};
\draw (-0.8,4.7) node {3};
\draw (-0.8,6.5) node {$\geq 4$};


\draw [fill=white] (12,0) rectangle (15.2,7.6);
\draw [fill=white] (12.2,1) rectangle (12.8,2.4);
\draw [fill=white] (12.2,2.6) rectangle (12.8,3.9);
\draw [fill=white] (12.2,4.1) rectangle (12.8,5.4);
\draw [fill=white] (12.2,5.6) rectangle (12.8,7.4);

\draw [fill=white] (13,0.2) rectangle (15,0.8);
\draw [fill=white] (13,1) rectangle (15,2.4);
\draw [fill=white] (13,2.6) rectangle (15,3.9);
\draw [fill=white] (13,4.1) rectangle (13.9,5.4);
\draw [fill=white] (13,5.6) rectangle (13.9,7.4);
\draw [fill=white] (14.1,4.1) rectangle (15,5.4);
\draw [fill=white] (14.1,5.6) rectangle (15,7.4);

\draw (13.6,-0.8) ellipse (1.6cm and 0.6cm);
\draw (13.4,-0.8) ellipse (0.5cm and 0.5cm);
\draw (14.5,-0.8) ellipse (0.5cm and 0.4cm);

\draw (11.2,3.8) ellipse (0.6cm and 3.8cm);
\draw (11.2,1.7) ellipse (0.4cm and 0.7cm);
\draw (11.2,3.2) ellipse (0.5cm and 0.7cm);
\draw (11.2,4.7) ellipse (0.5cm and 0.7cm);
\draw (11.2,6.5) ellipse (0.4cm and 0.9cm);

\path (12.5,0.5) coordinate (a);

\draw (a) [fill=white] circle (\vr);
\draw (12.5,-0.8) [fill=white] circle (\vr);
\draw (11.2,0.5) [fill=white] circle (\vr);

\draw (12.5,-1.5) node {$g$};
\draw (10.7,0.5) node {$h$};
\draw (13.5,-1.8) node {$G$};
\draw (10.3,3.8) node {$H$};
\draw (13.7,7.9) node {$G\diamond H$};
\draw (13.4,-0.8) node {$U$};
\draw (14.5,-0.8) node {$R$};

\draw (12.5,1.7) node {1};
\draw (14,0.5) node {1};
\draw (14,1.7) node {1};
\draw (13.4,4.7) node {3};
\draw (12.5,3.2) node {2};
\draw (12.5,6.5) node {3};
\draw (12.5,4.7) node {3};
\draw (14.5,4.7) node {2};
\draw (14.5,6.5) node {2};
\draw (13.4,6.5) node {3};
\draw (14,3.2) node {2};
\draw (11.2,3.2) node {2};
\draw (11.2,1.7) node {1};
\draw (11.2,4.7) node {3};
\draw (11.2,6.5) node {$\geq 4$};

\end{tikzpicture}
\end{center}

\caption{The distance from $(g,h)$ in $G\diamond H$ when $h$ is not universal and ($g$ is not universal on the left scheme and  $g$ is universal on the right scheme).}
\label{scheme}
\end{figure}
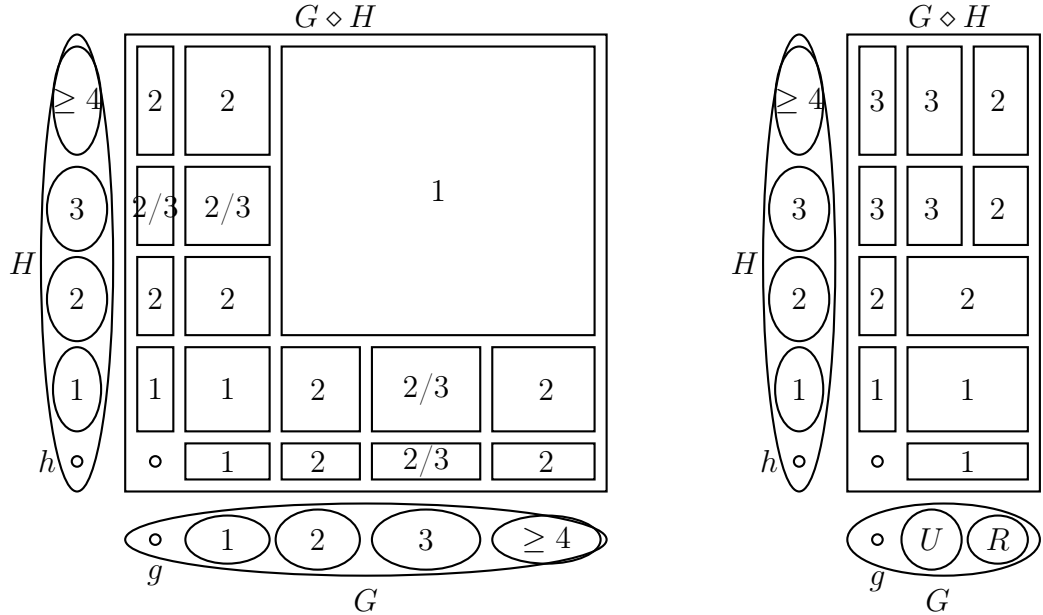


\section{Strong metric dimension of $G\diamond H$}

In this section we observe strong metric dimension of modular product $G\diamond H$. Let us first consider the case when one factor, say $H$, is complete. By (\ref{complete}) and Corollary 2.16 from \cite{Strong-strong} or Theorem 12 from \cite{strong-lex} we have
$${\rm dim}_s(G\diamond K_t)={\rm dim}_s(G\boxtimes K_t)={\rm dim}_s(G\circ K_t)=(t-1)n(G)+{\rm dim}_s(G).$$
So, we may assume in what follows that no factor is complete.

For the rest we first describe all edges of $E((G\diamond H)_{SR})$. Later we demonstrate how to obtain $\beta((G\diamond H)_{SR})$ for several infinite graph families in three separate subsections, depending on different types of edges that belong to $(G\diamond H)_{SR}$. By Theorem \ref{good} this equals to ${\rm dim}_s(G\diamond H)$ for the mentioned graph families. We start with an observation applicable to the modular product in many cases.

\begin{observation}\label{obsdiam}
Let a graph $X$ have diameter two. Then, $uv\in E(X_{SR})$ if and only if $d_X(u,v)=2$ or $N_X[u]=N_X[v]$ and $u\neq v$.
\end{observation}

Due to Corollary \ref{diam2}, $G\diamond H$ often has diameter two. But, there are exceptions described by (\ref{dim3eq1}) and symmetrically by (\ref{dim3eq2}). We will show that the exceptions play an important role in $(G\diamond H)_{SR}$. Recall that a vertex $u$ of a graph $X$ is called a \emph{boundary vertex} if there is a vertex $v$ of $X$ such that $d_X(u,v)=\diam(X)$. The following definition will simplify some statements.

\begin{definition}
Given a condition $P$ on vertices $(g,h)$ and $(g',h')$ in the modular product $G\diamond H$ (identified with $H\diamond G$), by the \textbf{mirror} of the condition $P$, we mean the condition obtained from $P$ by applying a transposition to symbols $g$ and $h$, to symbols $g'$ and $h'$, and to symbols $G$ and $H$.
\end{definition}

\begin{theorem}\label{SRedges}
For non-complete graphs $G$ and $H$ such that at least one of the two is not the disjoint union of two cliques, suppose $G\diamond H$ have diameter three. Then, for vertices $(g,h)$ and $(g',h')$ of $G\diamond H$, $(g,h)(g',h')\in E((G\diamond H)_{SR})$ if and only if the two vertices satisfy at least one of the following conditions:
\begin{description}
\item[$(i)$] $(g,h)$ and $(g',h')$ are distinct twins of $G\diamond H$;
\item[$(ii)$] $d_{G\diamond H}((g,h),(g',h'))=2$, and both $(g,h)$ and $(g',h')$ are not boundary vertices in $G\diamond H$;
\item[$(iii)$] $d_{G\diamond H}((g,h),(g',h'))=3$;
\item[$(iv)$] $g$ and $g'$ are universal in $G$, also $d_H(h,h')=2$, also $hh'\in E(H_{SR})$; or the mirror of the preceding condition;
\item[$(v)$] $g$ is universal but $g'$ is not universal in $G$, also $d_H(h,h')=2$ and $d_H(h,h'')\leq 2$ for every $h''\in N_H[h']$, also $h'$ does not belong to any $\gamma_H$-pair; or the mirror of the preceding condition.
\end{description}
\end{theorem}

\begin{proof}

$(\Leftarrow)$ Recall that $(g,h)(g',h')\in E((G\diamond H)_{SR})$ exactly when a geodesic between $(g,h)$ and $(g',h')$ cannot be extended to a longer geodesic. It is clear that each of conditions (i), (ii), (iii), and (iv) implies the maximality of any geodesic between the two vertices. \\

For $(v)$, we may assume that $g$ and $g'$ are universal in $G$ and that $d_H(h,h')=2$ and $hh'\in E(H_{SR})$. By (\ref{moddist}), we have $d_{G\diamond H}((g,h),(g',h'))=2$: $N_G[g]=N_G[g']$ and $d_H(h,h')=2$ imply $d_{G\diamond H}((g,h),(g',h'))>1$; (\ref{dim3eq1}) fails as $d_H(h,h')=2$ and (\ref{dim3eq2}) fails as $d_G(g,g')\leq 1$. If $(g_0,h_0)\in N_{G\diamond H}((g,h))-N_{G\diamond H}((g',h'))$, then $2\leq d_H(h',h_0)\leq 3$: since $g$ is universal, $(g_0,h_0)$ is joined with $(g,h)$ by a Cartesian or direct edge, implying $d_H(h,h_0)\leq 1$. Thus, $d_H(h',h_0)\leq 3$, since $d_H(h,h')=2$ by assumption. Note $(g_0,h_0)\neq (g',h')$ for they have different distances to $(g,h)$. Given $g'$ is universal, if $h_0\in N_H[h']$, then $(g_0,h_0)$ is adjacent to $(g',h')$. Thus, $d_H(h',h_0)\geq 2$. If $d_H(h_0,h')=3$, then the edge $h_0h$ extends any $h,h'$-geodesic in $H$, contradicting the assumption $hh'\in E(H_{SR})$. Hence, $d_H(h',h_0)=2$, and $d_{G\diamond H}((g',h'),(g_0,h_0))=2$ follows by (\ref{moddist}): condition (\ref{dim3eq1}) fails since $d_H(h',h_0)=2$; condition (\ref{dim3eq2}) fails since $d_G(g',g_0)\leq 1$, as $g'$ is universal in $G$. Therefore, a $(g,h),(g',h')$-geodesic cannot be extended at $(g,h)$. By symmetry of the conditions, a $(g,h),(g',h')$-geodesic cannot be extended at $(g',h')$. Thus, $(g,h)(g',h')\in E((G\diamond H)_{SR})$.\\

For $(vi)$, let us assume that $g$ is universal but $g'$ is not universal in $G$, also $d_H(h,h')=2$ and $d_H(h,h'')\leq 2$ for every $h''\in N_H[h']$, also $h'$ does not belong to any $\gamma_H$-pair; the mirror statement can be symmetrically argued. By (\ref{moddist}) we have $d_{G\diamond H}((g,h),(g',h'))=2$: $d_{G\diamond H}((g,h),(g',h'))>1$, because $d_G(g,g')=1$ but $d_H(h,h')=2$; $d_{G\diamond H}((g,h),(g',h'))<3$ because (\ref{dim3eq1}) fails since $d_H(h,h')=2$ and (\ref{dim3eq2}) fails since $d_G(g,g')=1$. Suppose there is a vertex $(x,y)\in G\diamond H$ for which $d_{G\diamond H}((g',h'),(x,y))=3$, then either (\ref{dim3eq1}) or (\ref{dim3eq2}) must hold for the pair $(x,y)$ and $(g',h')$. The universality of $g$ in $G$ implies that $\diam(G)=2$ and hence $d_G(x,g')\not\geq 3$; thus, (\ref{dim3eq2}) fails and (\ref{dim3eq1}) must hold. Since $g'$ is not universal (and $N_G[x]=N_G[g']$), $\{y,h'\}$ must form a $\gamma_H$-pair, contradicting the assumption on $h'$. Thus, $(g',h')$ is not a boundary vertex, and no $(g',h'),(g,h)$-geodesic can be extended at $(g,h)$. On the other hand, suppose a $(g,h),(x,y)$-geodesic of length three passes through $(g',h')$, then $(g',h')$ must be adjacent to $(x,y)$. Since $g$ is universal in $G$ (and thus $d_G(g,x)\not\geq 3$), (\ref{dim3eq1}) must hold for the pair $(g,h)$ and $(x,y)$, implying that $d_H(h,y)\geq 3$ and that $x$ is universal in $G$. Thus, we have $g'\in N_G(x)$ and hence $h'\in N_H[y]$ or, equivalently, $y\in N_H[h']$. Now, $d_H(h,y)\geq 3$ and $y\in N_H[h']$ contradict the hypothesis that $d_H(h,h'')\leq 2$ for every $h''\in N_H[h']$. Therefore, no $(g',h'),(g,h)$-geodesic can be extended at $(g',h')$.\\

$(\Rightarrow)$ Now, we will show that two vertices $(g,h)$ and $(g',h')$ being MMD in $G\diamond H$ (i.e., $(g,h)(g',h')\in E((G \diamond H)_{SR})$) implies the satisfaction of one of the five conditions. So, hereafter in this proof, assume $(g,h)(g',h')\in E((G \diamond H)_{SR})$. \\

Clearly, when $d_{G\diamond H} ((g,h),(g',h'))=1$, condition $(i)$ holds by Observation~\ref{twins2}, and condition $(iii)$ is $d_{G \diamond H}((g,h),(g',h'))=3$. Thus, we may assume that $d_{G \diamond H}((g,h), (g',h'))=2$. Condition $(ii)$ holds if neither of $(g,h)$ and $(g',h')$ is a boundary vertex. Otherwise, either $(g,h)$ or $(g',h')$ is a boundary vertex in $G\diamond H$. WLOG, let us assume that $(g,h)$ is a boundary vertex with respect to vertex $(x,y)$; i.e., $d_{G \diamond H}((g,h),(x,y))=3$. \\

By Theorem~\ref{dist3}, the condition that either $g$ and $x$ are universal in $G$ or $\{h,y\}$ is a $\gamma$-pair in $H$, or the mirror condition thereof, must hold. \\

First, let us assume that $\{h,y\}$ is a $\gamma_H$-pair; the case of $\{g,x\}$ being a $\gamma_G$-pair can be symmetrically argued. By condition (\ref{dim3eq1}) of Theorem~\ref{dist3}, we have $N_G[g]=N_G[x]$. \\

Assuming $hh'\in E(H)$, we have $gg'\notin E(G)$, for $d_{G \diamond H}((g,h),(g',h'))=1$ otherwise. Since $N_G[x]=N_G[g]$, $xg'\notin E(G)$. Notice $yh'\notin E(H)$, if $hh'\in E(H)$ and $\{h,y\}$ is a $\gamma_H$-pair. Thus, the vertices $(g',h')$ and $(x,y)$ are adjacent in $G\diamond H$; hence, $(g,h)$ and $(g',h')$ are not MMD in $G\diamond H$, contradicting the overall hypothesis. \\

We may thus assume $hh'\not\in E(H)$, and hence $yh'\in E(H)$, since $\{h,y\}$ is a $\gamma_H$-pair; we also have $gg'\in E(G)$, since
$d_{G \diamond H}((g,h),(g',h'))=1$ otherwise. Then, by $N_G[g]=N_G[x]$, we have $xg'\in E(G)$. Thus, again, $(x,y)$ extends any $(g,h),(g',h')$-geodesic in $G\diamond H$ at $(g',h')$, contradicting the overall hypothesis that $(g,h)$ is MMD with $(g',h')$ in $G\diamond H$. \\

Thus, a vertex $(u,v)$ MMD with another vertex at distance two in $G\diamond H$ may not have a factor belonging to a $\gamma$-pair. $(\maltese)$\\

We may assume that $g$ and $x$ are universal in $G$; the mirror case can be symmetrically argued. Since $d_{G \diamond H}((g,h), (g',h'))=2$, neither $h$ nor $h'$ can be universal in $H$. \\

Suppose $d_H(h,h')>2$; then, we have $d_{G \diamond H} ((g,h),(g,h'))=3$ by Theorem~\ref{dist3}, and thus $(g,h')\neq (g',h')$ and hence $g\neq g'$. But, the universality of $g$ in $G$ yields the adjacency of $(g,h')$ and $(g',h')$, enabling an extension of a $(g,h),(g',h')$-geodesic of length two to a $(g,h),(g,h')$-geodesic of length three, contradicting the overall hypothesis that $(g,h)$ is MMD with $(g',h')$ in $G\diamond H$. \\

We may thus assume $d_H(h,h')=2$, and consider two further subcases regarding the universality of $g'$ in $G$. \\

Suppose $g'$ is universal in $G$ and $hh'\not\in E(H_{\rm SR})$. Assuming WLOG that $h''\in N_H(h')$ and $d_H(h,h'')=3$, we have an edge between $(g',h')$ and $(g',h'')$, whereas $d_{G \diamond H} ((g,h),(g',h''))=3$ by Theorem~\ref{dist3}. The extension of a $(g,h),(g',h')$-geodesic to a $(g,h),(g',h'')$-geodesic again contradicts the overall hypothesis that $(g,h)$ is MMD with $(g',h')$ in $G\diamond H$. Thus, the universality of $g'$ in $G$ implies condition $(iv)$. \\

Finally, suppose $g'$ is not universal in $G$. Then, $(g',h')$ is not a boundary vertex by $(\maltese)$ and Theorem~\ref{dist3}, and hence any $(g,h),(g',h')$-geodesic can only be extended at $(g',h')$. Suppose, for the sake of contraposition, that $h'$ has a neighbor $h''$ in $H$ such that $d_H(h,h'')=3$. Then $(g,h'')$ extends a $(g,h),(g',h')$ geodesic at $(g',h')$. Note that $(g',h')$ is adjacent to $(g, h'')$ via a direct edge ($g$ is universal and $h'\in N_H(h'')$); also note that $d_{G \diamond H} ((g,h),(g,h''))=3$ by Theorem~\ref{dist3}. Therefore, the overall hypothesis $(g,h)(g',h')\in E((G \diamond H)_{SR})$, together with the other conditions of condition $(v)$, implies that $d_H(h,h'')\leq 2$ for all $h''\in N_H[h']$ and $(v)$ follows.
\end{proof}

Some of the conditions $(i)$-$(v)$ of Theorem \ref{SRedges} are not always present in factors $G$ and $H$. In particular for $(i)$, it is well possible that there are no twins in $G\diamond H$. Due to Theorem \ref{twins1} there exists no different twins in $G\diamond H$ when there are no different twins in factors and there is no $\gamma_G$-pair in $G$ or no $\gamma_H$-pair in $H$. Also $(iv)$ and $(v)$ are rarely fulfilled. In particular for $(iv)$, notice that $uv\in E(G_{SR})$ for any different false twins $u,v\in V(G)$. If $G\diamond H$ has diameter at most two, then $(iii)$ and $(v)$ does not hold.  Such graphs were described by Corollary \ref{diam2} and we continue with them.


\subsection{${\rm diam}(G\diamond H)=2$}

First we describe the edge set of $(G\diamond H)_{SR}$ when ${\rm diam}(G\diamond H)=2$ and then illustrate how to obtain the strong metric dimension on three families of graphs.

\begin{theorem}\label{easy}
Let $G$ and $H$ be graphs. If ${\rm diam}(G\diamond H)=2$, then
$$E((G\diamond H)_{SR})=TW(G\diamond H)\cup E(\overline{G}\Box\overline{H})\cup E(G\times\overline{H})\cup E(\overline{G}\times H).$$
\end{theorem}

\begin{proof}
We use Observation \ref{obsdiam} because ${\rm diam}(G\diamond H)=2$. Set $TW(G\diamond H)$ deals with twins. By the definition of $E(G\diamond H)$ a vertex $(g,h)\in V(G\diamond H)$ is non-adjacent to all vertices $(g,h')$ for every $h'\in\overline{N}_H[h]$, to all vertices $(g',h)$ for every $g'\in\overline{N}_G[g]$, to all vertices $(g',h')$ for every $g'\in N_G(g)$ and every $h'\in\overline{N}_H[h]$ and to all vertices $(g',h')$ for every $g'\in \overline{N}_G[g]$ and every $h'\in N_H(h)$. The first two conditions describe all the edges in $\overline{G}\Box\overline{H}$ as we traverse over all vertices $(g,h)\in V(G\diamond H)$. Edges of $G\times\overline{H}$ are obtained by the third condition for any $(g,h)\in V(G\diamond H)$ and the last condition yields the edges of $\overline{G}\times H$ for any $(g,h)\in V(G\diamond H)$.
\end{proof}

Next we use Theorem \ref{good} to obtain exact results on strong metric dimension.

\begin{proposition}
For integers $s\geq t\geq 2$ we have 
\begin{equation*}
{\rm dim}_s(K_{1,s}\diamond K_{1,t})=\left\{
\begin{array}{ccc}
st+s-1 & : & t=2 \\
st+s & : & t>2\\
\end{array}%
\right. .
\end{equation*}
\end{proposition}

\begin{proof}
By Theorem \ref{good} we have ${\rm dim}_s(K_{1,s}\diamond K_{1,t})=\beta((K_{1,s}\diamond K_{1,t})_{SR})$. Therefore we derive $\beta((K_{1,s}\diamond K_{1,t})_{SR})$.
Let $x$ and $y$ be the universal vertices of $K_{1,s}$ and $K_{1,t}$, respectively, and let $X=\{x_1,\dots,x_s\}$ and $Y=\{y_1,\dots,y_t\}$ be the other vertices of $K_{1,s}$ and $K_{1,t}$, respectively.  There are no different twin vertices in $K_{1,s}$ and in $K_{1,t}$ and we have $TW(K_{1,s}\diamond K_{1,t})=\emptyset$. Graph $\overline{K_{1,s}}$ contains an isolated vertex $x$ and a clique on $X$ and $\overline{K_{1,t}}$ contains an isolated vertex $y$ and a clique on $Y$. Both cliques on $X$ and $Y$ yield a subgraph $F$ of $(K_{1,s}\diamond K_{1,t})_{SR}$ isomorphic to $K_s\Box K_t$ by Theorem \ref{easy}. By the same theorem we also have
$$N_{(K_{1,s}\diamond K_{1,t})_{SR}}((x,y_j))=\{(x_i,y_j):i\in [s],j\in[t]-\{j\}\}\cup \{(x,y_j):j\in[t]-\{j\}\}$$
and
$$N_{(K_{1,s}\diamond K_{1,t})_{SR}}((x_i,y))=\{(x_i,y_j):i\in [s]-\{i\},j\in [t]\}\cup \{(x_i,y):i\in[s]-\{i\}\}.$$
In $F$ we have a collection of complete graphs $K_s$ induced by $X_j=\{(x_i,y_j):i\in[s]\}$ for every $j\in [t]$. So, every vertex cover contains at least $st-t$ vertices from $F$. In addition, $X\times\{y\}$ and $\{x\}\times Y$ induce complete graphs $K_s$ and $K_t$, respectively, and every vertex cover contains at least $s-1$ vertices of $X\times\{y\}$ and $t-1$ vertices of $\{x\}\times Y$. So, $\beta((K_{1,s}\diamond K_{1,t})_{SR})\geq st+s-2$. Similar to $X_i$ we define $Y_i=\{(x_i,y_j):j\in[t]\}$ for every $i\in [s]$. If a vertex of $\{x\}\times Y$ (or $X\times \{y\}$) is not in a vertex cover, say $(x,y_1)$ (or $(x_1,y)$) is not, then every vertex from $\cup_{i=2}^t$ (or $\cup_{i=2}^sY_i$) must be in this vertex cover. This means that $\beta((K_{1,s}\diamond K_{1,t})_{SR})\geq st+s$ when $t>2$ and $\beta((K_{1,s}\diamond K_{1,t})_{SR})\geq st+s-1$ when $t=2$.  

For the other inequality notice that set $A=(V(G)\times V(H))-\{(x,y),(x_i,y_i):i\in[t]\}$ covers all the above mentioned edges of $(K_{1,s}\diamond K_{1,t})_{SR}$ when $t>2$, which yields $\beta((K_{1,s}\diamond K_{1,t})_{SR})\leq st+s$. Similar set $A_1=(V(G)\times V(H))-\{(x,y),(x_1,y),(x,y_1),(x_1,y_1)\}$ is a vertex cover of $(K_{1,s}\diamond K_{1,t})_{SR}$ for $t=2$ and $\beta((K_{1,s}\diamond K_{1,t})_{SR})\leq st+s-1$ follows.
\end{proof}

We continue with complements of cycles and notice that ${\rm diam} (\overline{C}_n)=2$ for $n\geq 5$. Also, there are no twins in $\overline{C}_s\diamond \overline{C}_t$ by Theorem \ref{twins1} and we need to take care only of edges that at distance two by Observation \ref{obsdiam}. We will denote the cycles by $C_s=u_1\dots u_su_1$ and $C_t=v_1\dots v_tv_1$ and the summation in the first index is on modulo $s$ and in the second index is on modulo $t$ for every vertex of $\overline{C}_s\diamond \overline{C}_t$  and of $C_s\diamond C_t$ in what follows. We will also use term $j$-th \emph{row} for the vertices $R^j=\{(u_i,v_j):i\in[s]\}$ for every $j\in [t]$ and $i$-th \emph{column} for the vertices $C^i=\{(u_i,v_j):j\in[t]\}$ for every $i\in [s]$.

\begin{proposition}
For integers $s,t\geq 5$, $\max\{s,t\}\geq 6$, we have ${\rm dim}_s(\overline{C}_s\diamond \overline{C}_t)=st-\left\lfloor \frac{s}{2}\right\rfloor \left\lfloor \frac{t}{2}\right\rfloor$. In addition, ${\rm dim}_s(\overline{C}_5\diamond \overline{C}_5)=20$.
\end{proposition}

\begin{proof}
For $s=t=5$ notice that $V(\overline{C}_5\diamond \overline{C}_5)-\{(u_1,v_1)(u_2,v_2)(u_3,v_3)(u_4,v_4)(u_5,v_5)\}$ is a vertex cover. We checked by computer that there exists no vertex cover with smaller cardinality. Suppose next that $\max\{s,t\}\geq 6$. We have  $\overline{N}_{\overline{C}_s\diamond \overline{C}_t}(u_k,v_{\ell})=\{(u_k,v_{\ell -1}),(u_k,v_{\ell +1}),(u_{k-1},v_{\ell}),$ $(u_{k+1},v_{\ell})\}\cup\{(u_i,v_j):(|i-k|=1 \wedge |j-\ell|\geq 2)\vee (|i-k|\geq 2 \wedge |j-\ell|=1)\}$. By Observation \ref{obsdiam} we have $N_{(\overline{C}_s\diamond \overline{C}_t)_{SR}}(u_k,v_{\ell})=\overline{N}_{\overline{C}_s\diamond \overline{C}_t}(u_k,v_{\ell})$. Let $A=\{(u_{2i},v_{2j}):i\in \left[ \left\lfloor \frac{s}{2}\right\rfloor \right], j\in \left[ \left\lfloor \frac{t}{2}\right\rfloor \right] \}$. We will show that $B=(V(\overline{C}_s)\times V(\overline{C}_t))-A$ is a vertex cover set of $(\overline{C}_s\diamond \overline{C}_t)_{SR}$ and with that ${\rm dim}_s(\overline{C}_s\diamond \overline{C}_t)\leq st-\left\lfloor \frac{s}{2}\right\rfloor \left\lfloor \frac{t}{2}\right\rfloor$ by Theorem \ref{good}. No two vertices from $A$ are adjacent in $(\overline{C}_s\diamond \overline{C}_t)_{SR}$. Hence $A$ is an independent set and $B$ is a vertex cover of $(\overline{C}_s\diamond \overline{C}_t)_{SR}$ and ${\rm dim}_s(\overline{C}_s\diamond \overline{C}_t)\leq st-\left\lfloor \frac{s}{2}\right\rfloor \left\lfloor \frac{t}{2}\right\rfloor$ follows. 

Let now $B$ be a $\beta((\overline{C}_s\diamond \overline{C}_t)_{SR})$-set and $|B|={\rm dim}_s(\overline{C}_s\diamond \overline{C}_t)$ by Theorem \ref{good}. We may assume that $t\geq s$. Every vertex $(u_i,v_j)$ is adjacent to $(u_i,v_{j-1}),(u_i,v_{j+1}),(u_{i-1},v_j)$ and $(u_{i+1},v_j)$ in $(\overline{C}_s\diamond \overline{C}_t)_{SR}$. So, at most every second vertex in a row (column) is outside of $B$. This gives at most $\left\lfloor \frac{s}{2}\right\rfloor$ vertices of each row and at most $\left\lfloor \frac{t}{2}\right\rfloor$ vertices of each column outside of $B$. We will describe several configurations of vertices outside of $B$ that are separated by columns that are completely in $B$. 

Suppose that there is a column with at least two vertices outside of $B$, say $(u_1,v_1),(u_1,v_i)\notin B$ for some $i\in\{3,\dots, t-1\}$ by a possible change of notation. If $i\notin\{3,t-1\}$ or there are at least three vertices outside of $B$, then every vertex of the second column and the $s$-th column is adjacent to at least one vertex outside of $B$ from the first column in $(\overline{C}_s\diamond \overline{C}_t)_{SR}$. This means that whole second column and whole $s$-th column are in $B$ and there are at most $\left\lfloor \frac{t}{2}\right\rfloor$ vertices of the first column outside of $B$. Assume that $B$ yields $a$ described configurations. 

Another possibility is that $i\in\{3,t-1\}$, say that $i=3$ and $(u_1,v_3)\notin B$ and the first column has exactly two vertices outside of $B$. Only $(u_2,v_2)$ from the second column and only $(u_s,v_2)$ from the $s$-th column can be outside of $B$. If both $(u_2,v_2),(u_s,v_2)\notin B$, then the third and the $(s-1)$-th column belong completely to $B$. Further, if only one of $(u_2,v_2)$ and $(u_s,v_2)$ is outside of $B$, say $(u_2,v_2)\notin B$, then $s$-th column must be completely in $B$ and in the third column we have at most two vertices $(u_3,v_1)$ and $(u_3,v_3)$ outside of $B$. In both cases we have a configuration of five columns with at most five vertices outside of $B$ and suppose that there is $b$ such configurations. 

A version, similar to the previous one, is when $(u_1,v_1)$, $(u_1,v_3)$ and $(u_2,v_2)$ are the only vertices of $s$-th, first, second and third column outside of $B$. This gives (at most) three vertices outside of $B$ in four columns and let there be $c$ such configurations in our strong resolving graph. We can also have more than three but at most $s=\min\{s,t\}$ consecutive columns with exactly one vertex outside of $B$. Say that there are all together $d$ such columns. All together we have at most
$$q=a\left\lfloor \frac{t}{2}\right\rfloor+5b+3c+d$$
vertices outside of $B$. If $b>0$ or $d>0$, then $q<\left\lfloor \frac{s}{2}\right\rfloor\left\lfloor \frac{t}{2}\right\rfloor$ and we obtain a contradiction with $|B|=st-q>st-\left\lfloor \frac{s}{2}\right\rfloor\left\lfloor \frac{t}{2}\right\rfloor$. The same contradiction is obtained when $c>0$ and $t\geq 8$ or $c=2$ and $t\in\{6,7\}$. Finally, if $c=1$ and $t\in\{6,7\}$ or $c=0$, then we get that $q=\left\lfloor \frac{s}{2}\right\rfloor\left\lfloor \frac{t}{2}\right\rfloor$ vertices are outside of $B$ and the proof is completed.  
\end{proof}

\begin{proposition}
For integers $s,t\geq 7$ we have ${\rm dim}_s(C_s\diamond C_t)=st-4\min\{\left\lfloor \frac{s}{3}\right\rfloor, \left\lfloor \frac{t}{3}\right\rfloor\}-r$, where
\begin{equation*}
r=\left\{
\begin{array}{ccc}
0 & : & \min\{s,t\}\equiv \{0,1\}\ ({\rm mod}\ 3) \\
1 & : & s=t \wedge \min\{s,t\}\equiv 2\ ({\rm mod}\ 3)\\
2 & : & s\neq t \wedge \min\{s,t\}\equiv 2\ ({\rm mod}\ 3)
\end{array}%
\right. .
\end{equation*}
\end{proposition}

\begin{proof}
We may assume that $\min\{s,t\}=s$.  Let $A=\{(u_{3j-1},v_{3j-1})(u_{3j},v_{3j-1})(u_{3j-1},v_{3j})(u_{3j},v_{3j}):j\in \left[ \left\lfloor \frac{s}{3}\right\rfloor \right] \}$. We will show that $B=(V(G)\times V(H))-A$ is a vertex cover set of $(\overline{C}_s\diamond \overline{C}_t)_{SR}$ and with that ${\rm dim}_s(\overline{C}_s\diamond \overline{C}_t)\leq st-4\min\{\left\lfloor \frac{s}{3}\right\rfloor, \left\lfloor \frac{t}{3}\right\rfloor\}$ by Theorem \ref{good}. Cycles $C_s$ and $C_t$ have no $\gamma$-pairs and no universal vertices. Hence, ${\rm diam}(C_t\diamond C_s)=2$ by Corollary \ref{diam2}. Therefore, $N_{(C_s\diamond C_t)_{SR}}(u_k,v_{\ell})=\{(u_i,v_j):(|i-k|\leq 1 \wedge |j-\ell|\geq 2)\vee (|i-k|\geq 2 \wedge |j-\ell|\leq 1)\}$ by Observation \ref{obsdiam}. Hence, no two vertices from $A$ are adjacent in $(C_s\diamond C_t)_{SR}$ and $B$ is a vertex cover of $(C_s\diamond C_t)_{SR}$. In addition, if $s=t$ and $s\equiv 2\ ({\rm mod}\ 3)$, then $(u_s,v_s)$ is also independent from all vertices from $A$ and the complement of $A\cup\{(u_s,v_s)\}$ is a vertex cover of $(C_s\diamond C_t)_{SR}$, which means that ${\rm dim}_s(C_s\diamond C_t)\leq st-4\min\{\left\lfloor \frac{s}{3}\right\rfloor \left\lfloor \frac{t}{3}\right\rfloor\}-1$. Finally, if $s\neq t$ and $\min\{s,t\}=s\equiv 2\ ({\rm mod}\ 3)$, then we can add $(u_s,v_s),(u_s,v_{s+1})$ to $A$ and this set is still independent. Its complement is therefore a vertex cover of $(C_s\diamond C_t)_{SR}$ and we have ${\rm dim}_s(C_s\diamond C_t)\leq st-4\min\{\left\lfloor \frac{s}{3}\right\rfloor, \left\lfloor \frac{t}{3}\right\rfloor\}-2$.

For the other inequality we first prove the following claim.

\begin{claim} \label{consec}
Every three consecutive rows (columns) of $(C_s\diamond C_t)_{SR}$ contain at most four independent vertices.
\end{claim}

\noindent\emph{Proof.} Without loss of generality we may observe first three rows of $(C_s\diamond C_t)_{SR}$. Let the vertex $(u_i,v_1)$ be in an independent set $A$ of the first three rows of $(C_s\diamond C_t)_{SR}$.  A vertex $(u_i,v_1)$ is non-adjacent only to $(u_{i-1},v_1)$ and $(u_{i+1},v_1)$ in the first row in $(C_s\diamond C_t)_{SR}$. Moreover, $(u_{i-1},v_1)$ and $(u_{i+1},v_1)$ are adjacent in $(C_s\diamond C_t)_{SR}$ and there are at most two independent vertices of each row and if there are two, then they are consecutive. The same holds also for columns by symmetry. Observe next the second row. If $(u_i,v_2)$ is in our independent set $A$, then only two consecutive vertices of $(u_{i-1},v_2),(u_i,v_2),(u_{i+1},v_2)$ can also be in $A$. Moreover, if $(u_i,v_1)\in A$, then $(u_{i-1},v_3),(u_i,v_3),(i_{i+1},v_3)\notin A$ as they are adjacent to $(u_i,v_1)$ in $(C_s\diamond C_t)_{SR}$. Notice, that all the other vertices of the third row are adjacent to $(u_i,v_2)$ in $(C_s\diamond C_t)_{SR}$. So, $(u_i,v_1),(u_i,v_2)\in A$ implies that no vertex of the third row is in $A$ and we have $|A|=4$ because $A=\{(u_i,v_1),(u_{i+1},v_1),(u_i,v_2),(u_{i+1},v_2)\}$ is a maximal independent set in the first three columns of $(C_s\diamond C_t)_{SR}$. If $(u_i,v_1),(u_{i+1},v_2)\in A$, then there is also an independent set $A=\{(u_i,v_1),(u_{i+1},v_2),(u_{i+2},v_3)\}$ that is maximal in  the first three rows, since $(u_{i+1},v_1)$ and $(u_i,v_2)$ are adjacent to $(u_{i+2},v_3)$, $(u_{i-1},v_1)$ and $(u_{i+3},v_3)$ are adjacent to $(u_{i+1},v_2)$ and, $(u_{i+2},v_2)$ and $(u_{i+1},v_3)$ are adjacent to $(u_i,v_1)$ in $(C_s\diamond C_t)_{SR}$. So, $|A|=3\leq 4$. Otherwise, if the second row has no vertex in $A$, then we can have at most two (consecutive) vertices in first and third row and claim is proved.~\smallqed

Recall that $\min\{s,t\}=s$ and let $k=\left\lfloor \frac{s}{3}\right\rfloor$. Now we split columns $C^1,\dots,C^s$ of $(C_s\diamond C_t)_{SR}$ into consecutive triples $T_i=C^{3i-2}\cup C^{3i-1}\cup C^{3i}$ for $i\in [k]$. Let now $B$ be a $\beta((C_s\diamond C_t)_{SR})$-set and $|B|={\rm dim}_s(C_s\diamond C_t)$ by Theorem \ref{good}.

If $s\equiv 0\ ({\rm mod}\ 3)$, then $|T_i\cap B|\geq 3t-4$ by Claim \ref{consec} for every $i\in \left[\left\lfloor \frac{s}{3}\right\rfloor\right]$. Hence,
\begin{equation}\label{comput}
|B|=\sum_{i=1}^k{|T_i\cap B|}\geq \sum_{i=1}^k{3t-4}=3kt-4k=st-4\min\left\{\left\lfloor \frac{s}{3} \right\rfloor,\left\lfloor \frac{t}{3} \right\rfloor\right\}.
\end{equation}

If $s\equiv 1\ ({\rm mod}\ 3)$, then we also use the structure of independent set described by Claim \ref{consec}. If we have, by possible change of notation, $(u_1,v_1),(u_1,v_2),(u_2,v_1),(u_2,v_2)\notin B$, then $B$ contains all vertices of the third and the last column by Claim \ref{consec}. Therefore, we can subtract at most so many vertices as in (\ref{comput}) from $B$ and the desired lower bound for $B$ follows again. Finally, let $s\equiv 2\ ({\rm mod}\ 3)$. The scenario with the most vertices outside of $B$ is the following: every triple $T_i$, $i\in [k]$, contains exactly four vertices outside of $B$ and no two consecutive columns are completely in $B$. We can again assume that $(u_1,v_1),(u_1,v_2),(u_2,v_1),(u_2,v_2)\notin B$. With this third and last column must be completely in $B$ by Claim \ref{consec} and  $(s-2)$-th column and row are completely in $B$. Hence $(s-1)$-th column and $(s-1)$th row also contain a vertex outside of $B$. If $s=t$, then $(u_{s-1},v_{s-1})\notin B$ and when $s\neq t$, then some two consecutive vertices from $(u_{s-1},v_{s-1}),(u_{s-1},v_s),\dots ,(u_{s-1},v_{t-1})$ are not in $B$. In first case we need to substract one more from (\ref{comput}) and in the second case two more from (\ref{comput}). In both cases the desired lover bound follows and the proof is completed.
\end{proof}


\subsection{One factor has a $\gamma$-pair and the other has no universal vertex}

Also in this case we can describe an analogue result to Theorem \ref{easy}. For that we need some additional terminology. Let set $P(G)$ contains all vertices of $G$ that belong to $\gamma_G$-pairs. Further we denote $G-P(G)$ by $G^-$ and $\overline{G}-P(G)$ by $\overline{G}^-$. We also define a $\gamma_G$-pair graph $GP(G)$ as $V(GP(G))=V(G)$ where $uv\in E(GP(G))$ when $\{u,v\}$ is a $\gamma_G$-pair.

\begin{theorem}\label{easy1}
If a graph $G$ has a $\gamma_G$-pair and a graph $H$ is without a universal vertex, then
$$E((G\diamond H)_{SR})=TW(G\diamond H)\cup E(GP(G)\Box GP(H))\cup E(\overline{G}^-\Box\overline{H}^-)\cup E(G^-\times\overline{H}^-)\cup E(\overline{G}^-\times H^-).$$
\end{theorem}

\begin{proof}
Only $(i)$, $(ii)$ and $(iii)$ of Theorem \ref{SRedges} can be fulfilled, because $G$ and $H$ have no universal vertices. Set $TW(G\diamond H)$ deals with $(i)$. Vertices that induce edges of $(G\diamond H)_{SR}$ according to $(ii)$ are different than vertices that induce edges with respect to $(iii)$ of Theorem \ref{SRedges}. By the same reasons as in the proof of Theorem \ref{easy} we can see that $E(\overline{G}^-\Box\overline{H}^-)\cup E(G^-\times\overline{H}^-)\cup E(\overline{G}^-\times H^-)$ describe all edges of $(G\diamond H)_{SR}$ with respect to $(ii)$ of Theorem \ref{SRedges}. Finally, $E(GP(G)\Box GP(H))$ contains exactly the edges of $(G\diamond H)_{SR}$ that fulfill $(iii)$ of Theorem \ref{SRedges} by (\ref{moddist}) as there are no universal vertices in $G$ and in $H$.
\end{proof}

In addition to above theorem, notice that endvertices of the last three sets of edges in $E((G\diamond H)_{SR})$ are disjunctive from the endvertices from $E(GP(G)\Box GP(H))$ by Theorem \ref{SRedges}.

Special family of graphs with a $\gamma_G$-pair are graphs where every vertex is in at least one $\gamma_G$-pair. The smallest example of such a graph is $C_6\cong K_{3,3}-M$, where $M$ is a perfect matching of $K_{3,3}$. It is easy to see that also $K_{n,n}^{-M}\cong K_{n,n}-M$, where $M$ is a perfect matching of $K_{n,n}$, is such a graph. If $V_1=\{x_1,\dots, x_n\}$ and $V_2=\{y_1,\dots, y_n\}$ form a partition of $V(K_{n,n})$ and $M=\{x_iy_i:i\in [n]\}$, then $\{x_i,y_i\}$ is a $\gamma_G$-pair for each $i\in [n]$. If we replace one vertex, say $x_1$, by a clique on $q_1$ vertices, then the obtained graph still has every vertex in a $\gamma$-pair, only that we have $q_1+n-1$ different $\gamma$-pairs. This remains also if other vertices of $V_1$ or $V_2$ are replaced by a clique. Let $K_{n,n}^{-M}(q_1,\dots,q_n,r_1,\dots,r_n)$ be a graph obtained from $K_{n,n}^{-M}$ where $x_i$ was replaced by a clique on $q_i\geq 1$ vertices and $y_i$ by a clique on $r_i\geq 1$ vertices for every $i\in [n]$.

As mentioned in the Preliminaries section, the relation twin is an equivalence relation. Suppose that graph $H$ has $t(H)$ different equivalent classes $[h_1],\dots, [h_{\ell}]$. Moreover, we may assume that we start with equivalence classes $[h_1],\dots [h_{k}]$, $k\in\{0,\dots,\ell\}$, where $h_i$ does not belong to a $\gamma_H$-pair. The rest of equivalence classes $[h_{k+1}],\dots, [h_{\ell}]$, they exists when $k<\ell$,  are ordered in such a way that $\{h_{k+i},h_{\ell-i+1}\}$ is a $\gamma_H$-pair for every $i\in \left[\frac{\ell-k}{2}\right]$. We denote $T_i(H)=[h_i]$ for $i\in [k]$ and $T_{k+i}(H)=[h_{k+i}]\cup [h_{\ell-i+1}]$ for $i\in \left[\frac{\ell-k}{2}\right]$ and let $t_i(H)=|T_i(H)|$ for every $i\in \left[k+\frac{\ell-k}{2}\right]$. We will use $k(H)=k+\frac{\ell-k}{2}$. Clearly, $k(H)=n(H)$ when $H$ has no different twins and is without a $\gamma_H$-pair. In particular, for $H\cong K_{n,n}^{-M}(q_1,\dots,q_n,r_1,\dots,r_n)$ we have $k(H)=n$ and $t_i(H)=q_i+r_i$ for every $i\in [n]$.

\begin{proposition}\label{ex}
For integer $n\geq 3$ and positive integers $q_1,\dots,q_n,r_1,\dots,r_n$ and a graph $H$ without a universal vertex we have
$${\rm dim}_s(K_{n,n}^{-M}(q_1,\dots,q_n,r_1,\dots,r_n)\diamond H)=\sum_{i=1}^n{\left(t_i(G)\sum_{j=1}^{k(H)}{t_j(H)}\right)}-nk(H).$$
\end{proposition}

\begin{proof}
Let $G\cong K_{n,n}^{-M}(q_1,\dots,q_n,r_1,\dots,r_n)$ and $H$ a graph without a universal vertex. Every vertex of $G\diamond H$ is a boundary vertex and we have no edges that fulfill $(ii)$ of Theorem \ref{SRedges} in $(G\diamond H)_{SR}$. Choose $g$ and $g'$ from $T_i(G)$, $i\in [n]$, and $h$ and $h'$ from $T_j(H)$, $j\in [k(H)]$. Notice that vertices $(g,h)$ and $(g',h')$ are either twins in $G\diamond H$ by Theorem \ref{twins1} or at distance three in $G\diamond H$ by (\ref{moddist}). In each case they are adjacent in $(G\diamond H)_{SR}$ by Theorem \ref{SRedges}. So, such vertices form a clique $Q_{i,j}$ of cardinality $t_i(G)t_j(H)$ for every $i\in [n]$ and $j\in [k(H)]$. Moreover, different cliques $Q_{i,j}$ and $Q_{k,\ell}$ have empty intersection. Because $\beta(K_m)=m-1$ and there are no other edges in $(G\diamond H)_{SR}$ by Theorem \ref{easy1}, we obtain that $\beta((G\diamond H)_{SR})$ equals to
$$\sum_{i=1}^n{\left(\sum_{j=1}^{k(H)} {\beta(Q_{i,j})}\right)}=\sum_{i=1}^n{\left(\sum_{j=1}^{k(H)} {(t_i(G)t_j(H)-1)}\right)}=\sum_{i=1}^n{\left(t_i(G)\sum_{j=1}^{k(H)} {t_j(H)}\right)}-nk(H).$$
By Theorem \ref{good} the result follows.
\end{proof}

Three direct consequences of above proposition follows.

\begin{corollary}
For integers $n\geq 3$ and a graph $H$ without a universal vertex we have
$${\rm dim}_s(K_{n,n}^{-M}\diamond H)=2n\sum_{j=1}^{k(H)}{t_j(H)}-nk(H).$$
\end{corollary}

\begin{corollary}
For integers $n,m\geq 3$ and positive integers $q_1,\dots,q_n,r_1,\dots,r_n,s_1,\dots,s_m,p_1,\dots,p_m$ we have
$${\rm dim}_s(K_{n,n}^{-M}(q_1,\dots,q_n,r_1,\dots,r_n)\diamond K_{m,m}^{-M}(s_1,\dots,s_m,p_1,\dots,p_m))=\sum_{i=1}^n{\left((q_i+r_i)\sum_{j=1}^m{(s_j+p_j)}\right)}-nm.$$
\end{corollary}

\begin{corollary}
For integers $n,m\geq 3$ we have ${\rm dim}_s(K_{n,n}^{-M}\diamond K_{m,m}^{-M})=3nm$.
\end{corollary}

\begin{proposition}\label{ex2}
For a graph $H$ without a universal vertex and without different twins we have
$${\rm dim}_s(P_5\diamond H)=4\sum_{j=1}^{k(H)}{t_j(H)}-2k(H)+\beta(\overline{H}^-).$$
\end{proposition}

\begin{proof}
Let $P_5=v_1v_2v_3v_4v_5$ and let $H$ a graph without a universal vertex. Clearly, $\{v_1,v_4\}$ and $\{v_2,v_5\}$ are $\gamma_{P_5}$-pairs and $P_5^-=v_3$ is a singleton. Therefore, $\overline{P}_5^-\Box\overline{H}^-\cong \overline{H}^-$, while $P_5^-\times\overline{H}^-$ and $\overline{P}_5^-\times H^-$ are graphs without edges. Set $TW(P_5\diamond H)$ is nonempty only when $H$ has a $\gamma_H$-pair by Theorem \ref{twins1}, because $H$ and $P_5$ have no different twins. Thus, edges in $E(\overline{P}_5^-\Box\overline{H}^-)\cup E(P_5^-\times\overline{H}^-)\cup E(\overline{P}_5^-\times H^-)$ have no endvertex in common with edges in $TW(P_5\diamond H)\cup E(GP(P_5)\Box GP(H))$. Therefore, $\overline{H}^-$ contributes to ${\rm dim}_s(P_5\diamond H)$ exactly $\beta(\overline{H}^-)$ by Theorem \ref{good}.

A subgraph of $(P_5\diamond H)_{SR}$ induced by the edges from $TW(P_5\diamond H)\cup E(GP(P_5)\Box GP(H))$ contains disjunctive cliques on $2t_j(H)$ vertices for every $j\in [k(H)]$ and every $\gamma_{P_5}$-pair by the same reasons as in Proposition \ref{ex}. This gives exactly $4\sum_{j=1}^{k(H)}{t_j(H)}-2k(H)$ to ${\rm dim}_s(P_5\diamond H)$ by Theorem \ref{good} since $P_5$ has exactly two disjunctive $\gamma_{P_5}$-pairs. Altogether the result follows.
\end{proof}

\begin{corollary}
For integer $r\geq 7$ we have ${\rm dim}_s(P_5\diamond P_r)={\rm dim}_s(P_5\diamond C_r)=3r-2$.
\end{corollary}

\begin{proof}
By Proposition \ref{ex2} we have ${\rm dim}_s(P_5\diamond P_r)=2r+\beta(\overline{P}_r^-)$ because $P_r$ has no $\gamma_{P_r}$-pairs and no different twins. It is easy to see that the maximum independent set of $\overline{P}_r^-$ contains two adjacent vertices of $P_r$. Hence $\beta(\overline{P}_r^-)=r-2$ and the result follows for $P_r$. The same conclusion can be done for $C_r$, $r\geq 7$.
\end{proof}


\subsection{One factor has a universal vertex and the other is arbitrary}

In this case all five conditions of Theorem \ref{SRedges} can occure. We ilustrate that on the following family of modular products. We define a familly of graphs $H(s,t,q)$ for integers $s\geq 1$, $t\geq 2$ and $q\geq 0$. We have $V(H(s,t,q))=X\cup Y\cup W\cup\{z\}$ for $X=\{x_1,\dots,x_s\}$, $Y=\{y_1,\dots,y_t\}$ and $W=\{w_1,\dots,w_q\}$. Edges are defined by $E(H(s,t,q))=\{x_iy_j:i\in [s],j\in [t]\}\cup \{zy_k:k\in [t-1]\}\cup\{zw_{\ell}:\ell\in [q]\}$. For the minimum values of parameters $s,t$ and $q$ we have $H(1,2,0)\cong P_4$ and similar $H(1,2,1)\cong P_5$. Notice that $H(s,t,q)$ is connected because $t\geq 2$ and $s\geq 1$ and that $\{y_t,z\}$ form a $\gamma_{H(s,t,q)}$-pair. In addition, if $s=1=q$, then also $\{x_1,w_1\}$ form a $\gamma_{H(1,t,1)}$-pair. We will omit this possibility. Vertices from $W$ are leaves with a common neighbor $z$ and therefore members of $W$ are false twins. Similar $X\cup Y$ induces $K_{s,t}$ and vertices from $X$ have no other neighbor, so they are false twins as well. False twins are also the vertices from $Y-\{y_t\}$. In addition, there are no different twins in $H(s,t,q)$.

Next we describe all edges of $(K_{1,r}\diamond H(s,t,q))_{SR}$ for integer $r\geq 2$ and $(s,t)\neq(1,1)$. Here is $v$ the universal vertex and $U=\{u_1,\dots,u_r\}$ the leaves of $K_{1,r}$. There are no different twins in $K_{1,r}\diamond H(s,t,q)$ by Theorem \ref{twins1} since there are no different twins in $K_{1,r}$ and in $H(s,t,q)$ and since there is no $\gamma_{K_{1,r}}$-pair. 

Vertices from $Q_X=\{v\}\times X$ and from $Q_W=\{v\}\times W$ induce cliques in $(K_{1,r}\diamond H(s,t,q))_{SR}$, because they fulfill $(iv)$ of Theorem \ref{SRedges}. In addition, there are also all edges between $Q_X$ and $Q_W$ in $(K_{1,r}\diamond H(s,t,q))_{SR}$, because they fulfill $(iii)$ of Theorem \ref{SRedges}. Hence $Q_X\cup Q_W$ is a clique with $q+s$ vertices in $(K_{1,r}\diamond H(s,t,q))_{SR}$ (type A edges). 
All vertices from $Q_W$ are at distance three to the vertex $(v,y_t)$ by (\ref{moddist}). Hence they are adjacent in $(K_{1,r}\diamond H(s,t,q))_{SR}$ by $(iii)$ of Theorem \ref{SRedges} (type B edges).
For all false twins $w_k$ of $w_i$ and all false twins $x_{\ell}$ of $x_j$ it holds $(v,w_i)(u_m,w_k),(v,x_j)(u_m,x_{\ell})\in E((K_{1,r}\diamond H(s,t,q))_{SR})$ for every $m\in[r]$, $i,k\in[q]$, $i\neq k$, and $j,\ell \in[s]$, $j\neq \ell$ due to the $(v)$ of Theorem \ref{SRedges} (type D edges).
The vertices from $Q_X$ and $Q_W$ are not adjacent to any other vertices in $(K_{1,r}\diamond H(s,t,q))_{SR}$, since such a vertex is either at distance $1$ to the respective vertex or at distance $2$ and it lies on a path of length $3$ that starts at the respective vertex.

Moreover, to finish with $(iii)$ of Theorem \ref{SRedges}, we have $(v,z)(v,y_t),(u_i,y_t)(u_i,z)\in E((K_{1,r}\diamond H(s,t,q))_{SR})$ for every $i\in [r]$ (type C edges) because they are at distance three as well in $K_{1,r}\diamond H(s,t,q)$ since $\{y_t,z\}$ is a $\gamma_{H(s,t,q)}$-pair. In addition, vertices $(u_i,y_t),(u_i,z),(v,z),(v,y_t)$, $i\in [r]$, are boundary vertices and they are not adjacent to any other vertices in $(K_{1,r}\diamond H(s,t,q))_{SR}$, since again such a vertex is either at distance $1$ to the respective vertex or at distance $2$ and it lies on a path of length $3$ that starts at the respective vertex. Therefore, we are done with boundary vertices.

We still need to describe which edges fulfill $(ii)$ of Theorem \ref{SRedges}. Let $A=X\cup (Y-\{y_t\})\cup W$. It is easy to see that vertices that are not boundary vertices belong to $U\times A$ and to $Q_Y=\{v\}\times (Y-\{y_t\})$, remember that we are omitting $s=q=1$. Vertices of $Q_Y$ form a clique in $(K_{1,r}\diamond H(s,t,q))_{SR}$ by $(ii)$ of Theorem \ref{SRedges} (type E edges). In addition, every vertex $(v,y_i)\in Q_Y$, $i\in [t-1]$, is also adjacent in $(K_{1,r}\diamond H(s,t,q))_{SR}$ to all vertices of $U\times W$ and to $U\times (Y-\{y_i, y_t\})$ by $(ii)$ of Theorem \ref{SRedges} (also type E edges), since they are not adjacent in $K_{1,r}\diamond H(s,t,q)$. Notice also that  $(v,y_i)\in Q_Y$, $i\in [t-1]$, is not adjacent to other vertices of $U\times A$ in $(K_{1,r}\diamond H(s,t,q))_{SR}$ as they are adjacent in $K_{1,r}\diamond H(s,t,q)$. Set $Q_a=U\times\{a\}$, $a\in A$, induces a clique in $(K_{1,r}\diamond H(s,t,q))_{SR}$ by the same reason (type F edges). Vertices from $\{u_i\}\times (X\cup W)$ (type G edges) and $\{u_i\}\times ((Y-\{y_t\})\cup W)$ (type H edges) induce cliques in $(K_{1,r}\diamond H(s,t,q))_{SR}$ since they are not adjacent in $K_{1,r}\diamond H(s,t,q)$. Finally, vertices $(u_i,x_j)$ and $(u_{\ell},y_k)$ are adjacent for $i,\ell\in[r]$, $i\neq\ell$ and $j\in[s]$ and $k\in[t-1]$, by the same reason again (type I edges). There are no other edges between vertices of $U\times A$ in $(K_{1,r}\diamond H(s,t,q))_{SR}$ because they are all adjacent in $K_{1,r}\diamond H(s,t,q)$.

With this knowledge, we can prove the following. We omit some small values for $s,t$ and $q$ to avoid too many cases.

\begin{proposition}
For integers $r,q\geq 3$, $s,t\geq 4$, $m_q=\min\{r,q\}$, $m_t=\min\{t-1,r-q\}$, $m_s=\min\{s,r-q\}$ and
\begin{equation*}
b=\left\{
\begin{array}{ccc}
r+2 & : & r\leq q+1 \\
r+1 & : & r=q+2 \vee (r\geq q+3 \wedge \max\{s+1,t\}\geq r)\\
q+m_s & : & r\geq q+3 \wedge t\leq s<r-1\\
q+m_t & : & r\geq q+3 \wedge s<t<r
\end{array}%
\right. ,
\end{equation*}%
we have ${\rm dim}_s(K_{1,r}\diamond H(s,t,q))=(s+t+q-1)r-b+r+q+s+t$.

\end{proposition}

\begin{proof}
We will use the notation introduced before and the following sets $B_q=\{(u_i,w_i):i\in[m_q]\}$, $B_q^-=B_q-\{(u_{m_q},w_{m_q})\}$, $B_t=\{(u_i,x_i):i\in[r-1]\}$, $B_s=\{(u_i,y_i):i\in[r-1]\}$, $B_x=\{(u_{q+i},x_i):i\in[m_s]\}$, $B_y=\{(u_{q+i},y_i):i\in[m_t]\}$, $B_2=\{(v,w_1),(u_r,w_1)\}$ and $B_3=\{(u_r,y_1),(u_r,x_1),(v,x_1)\}$. With this we can define the next set
\begin{equation*}
M=\left\{
\begin{array}{ccc}
B_q^-\cup B_3 & : & r\leq q \\
B_q\cup B_3 & : & r\in\{q+1,q+2\}\\
B_s\cup B_2 & : & r\geq q+3 \wedge s\geq \max\{r-1,t\}  \\
B_t\cup B_2 & : & r\geq q+3 \wedge t\geq \max\{r-1,s-1\} \\
B_q\cup B_x & : & r\geq q+3 \wedge t\leq s<r-1\\
B_q\cup B_y & : & r\geq q+3 \wedge r>t>s
\end{array}%
\right. .
\end{equation*}%
We will show that
$$B=((U\times A)\cup (U\times\{y_t\})\cup Q_W\cup Q_X\cup Q_Y\cup\{(v,y_t)\})-M$$
is a vertex cover of $(K_{1,r}\diamond H(s,t,q))_{SR}$ of cardinality $c=(s+t+q-1)r+r+q+s+t-b$. For the cardinality notice that $|Q_X|=s$, $|Q_W|=q$, $|\{(v,y_t)\}|=1$, $|U\times\{y_t\}|=r$, $|Q_Y|=t-1$, $|(U\times A)|=(s+t+q-1)r$ and $|M|=b$ and we have the desired result. For the vertex cover notice that $B\cup M\cup (V(K_{1,r})\times \{z\})$ are all vertices that are end-vertices of $(K_{1,r}\diamond H(s,t,q))_{SR}$. Moreover, vertices of $V(K_{1,r})\times \{z\}$ have degree one in $(K_{1,r}\diamond H(s,t,q))_{SR}$ and $B\cup M$ is a vertex cover since the neighbors of $V(K_{1,r})\times \{z\}$, that are $V(K_{1,r})\times \{y_t\}$, remain in $B\cup M$. They also remain in $B$. So, it is enough to show that vertices of $M$ are independent in $(K_{1,r}\diamond H(s,t,q))_{SR}$. For this $U\times (X\cup W)$ induces $K_r\Box K_{s+q}$ and  $U\times ((Y-\{y_t\})\cup W)$ induces $K_r\Box K_{t+q-1}$ in $(K_{1,r}\diamond H(s,t,q))_{SR}$. So, at most one vertex can miss in every copy of $K_r$ induced by $U\times \{a\}$, for $a\in A$ and at most one vertex in copy of $K_{s+q}$ induced by $\{u_i\}\times (X\cup W)$, $i\in [r]$, and at most one vertex in copy of $K_{t+q-1}$ induced by $\{u_i\}\times ((Y-\{y_t\})\cup W)$, $i\in [r]$. Clearly $M$ fulfills these properties. Notice also that $M$ sometimes contains $(v,x_1)\in B_3$. In such a case $(u_r,x_1)$ is the only vertex from $U\times X$ in $M$ and every edge of type D is still covered by $B$. Similar happens when $M$ contains $(v,w_1)\in B_2$. Hence, $B$ is a vertex cover of $(K_{1,r}\diamond H(s,t,q))_{SR}$ of cardinality $c$. By Theorem \ref{good} we have ${\rm dim}_s(K_{1,r}\diamond H(s,t,q))\leq c$.

Let now $C$ be any vertex cover set of $(K_{1,r}\diamond H(s,t,q))_{SR}$ of minimum cardinality. In what follows we analyze what happens if $C$ does not contain all vertices from $Q_X$ or from $Q_Y$ or from $Q_W$ or when it does contains all of them. We will obtain some conditions on $|C|$ which lead either to $|C|>c=|B|$, that is a contradiction with $|C|\leq |B|$, or to $|C|\geq c$, that is, to equality $|C|=c$ which yields the desired result by Theorem \ref{good}. We refer to the first option simply as ``a contradiction'' and to the second as ``the equality''.  Set $C$ contains at least $s+q-1$ vertices of $Q_X\cup Q_W$, because it is a clique of $(K_{1,r}\diamond H(s,t,q))_{SR}$. Similar $C$ contains at least $t-2$ vertices of $Q_Y$. Edges $(v,y_t)(v,z)$ and $(u_i,y_t)(u_i,z)$, $i\in [r]$, are leaves of $(K_{1,r}\diamond H(s,t,q))_{SR}$ and $C$ contains at least $r+1$ of their end-vertices.
So, for now $|C|\geq s+q+r+t-2$.

Suppose first that a vertex from $Q_Y$, say $(v,y_1)$, is outside of $C$. By type E edges every vertex of $U\times W$ must be in $C$. By type E and F edges at most one vertex from $U\times (Y-\{y_t\})$, say $(u_1,y_1)$, is outside of $C$. In such a case, at most one vertex from $U\times X$ is outside of $C$ by type I and type G edges and let this vertex be $(u_1,x_1)$. Hence, we have additional $(s+q+t-1)r-2$ vertices in $C$ which gives $|C|\geq (s+q+t-1)r+s+q+r+t-4$. This is a contradiction because $b\geq 5$ in all four possibilities.

We continue with $(v,y_1)\notin C$ but $U\times (Y-\{y_t\})\subset C$. If also a vertex from $Q_X$ is outside of $C$, say $(v,x_1)$, then only one additional vertex from $U\times X$ can be outside of $C$ by type D and type F edges. Thus at most three additional vertices are outside of $C$ which is a contradiction again. So, $Q_X\subset C$ and there can be at most $\min\{s,r\}$ vertices from $U\times X$ outside of $C$ by type F and type G edges. This gives at most $1+\min\{s,r\}$ vertices from $Q_x\cup Q_y\cup (U\times A)$ outside of $C$ and $|C|\geq (s+q+t-1)r+s+q+r+t-(1+\min\{s,r\})$ follows. In all the cases we obtain a contradiction because $b>1+\min\{s,r\}$.

From now on we may assume that $Q_Y\subseteq C$. Assume that there exists a vertex from $Q_X$ outside of $C$, say $(v,x_1)$. Edges of type D and type F yields that at most one vertex from $U\times X$ is outside of $C$, say $(u_r,x_1)$ is not in $C$. Moreover, edges of type I and type H allows only one additional vertex from $U\times (Y-\{y_t\})$ outside of $C$, say that $(u_r,y_1)$ is such. (Notice that these three vertices form set $B_3$.) By type G, type H and type F edges at most one vertex from $\{u_i\}\times W$ is outside of $C$ for every $i\in [\min\{q,r-1\}]$, say that $(u_i,w_i)\notin C$. (Notice that these vertices form $B_q^-$ when $r\leq q$ and $B_q$ when $r>q$.) With this $|C|\geq (s+q+t-1)r+s+q+r+t-(\min\{q,r-1\}+3)$. If $r\leq q+1$, then $|C|\geq |B|$ and  the equality follows. If $r=q+2$, then $\min\{q,r-1\}+3=q+3=r+1=b$ which means that $|C|\geq |B|$ and  equality holds again. We are left with $r\geq q+3$ where $\min\{q,r-1\}+3=q+3$ holds. If $\max\{s+1,t\}\geq r$, then $b=r+1>q+3$, a contradiction.
If $t\leq s<r-1$, then $b$ equals to $q+s$ or to $r$, a contradiction when $r>q+3$ and the equality when $r=q+3$.  If $s<t<r$, then $t\geq 5$ and $b$ equals to $q+t-1$ or to $r$, a contradiction.

Next we may assume that $Q_X\cup Q_y\subseteq C$. Assume also that there exists a vertex from $Q_W$ outside of $C$, say $(v,w_1)$. Edges of type D and type F yields that at most one vertex from $U\times W$ is outside of $C$, say $(u_r,w_1)$ is not in $C$. If there is a vertex from $U\times (Y-\{y_t\})$ and a vertex from $U\times X$ outside of $C$, then they must be of the form $(u_i,y_j)$ and $(u_i,x_k)$ by type I edges for some $i\in[r-1]$, $j\in[t-1]$ and $k\in [s]$. Because $b\geq 5$ this forces a contradiction. So, we may assume that $U\times (Y-\{y_t\})\subset C$ or $U\times X\subset C$. Let first $U\times (Y-\{y_t\})\subset C$. There are at most $\min\{r-1,s\}$ vertices of $U\times X$ outside of $C$ by type G and type F edges. This gives $|C|\geq (s+q+t-1)r+s+q+r+t-(\min\{r-1,s\}+2)$. If $\min\{r-1,s\}=r-1$, then we need to compare $r+1$ only with first two lines of $b$ for different values of $r,s,t,q$. If $r\leq q+1$, then we have $|C|>|B|=c$ since $b=r+2$, a contradiction. In the second line we have $b=r+1$ and the equality follows. So, we may assume that $\min\{r-1,s\}=s<r-1$ and we compare $s+2<r+1$ with $b$. In first two lines we get directly a contradiction. In the last two lines we have the equality only when $s=r-2$ and otherwise a contradiction.
It remains that $U\times X\subset C$. There are at most $\min\{r-1,t-1\}$ vertices of $U\times (Y-\{y_t\})$ outside of $C$ by types H and type F edges. This gives $|C|\geq (s+q+t-1)r+s+q+r+t-(\min\{r-1,t-1\}+2)$. If $\min\{r-1,t-1\}=r-1$, then we need to compare $r+1$ only with first two lines of $b$ as before. If $r\leq q+1$, then we have $|C|>|B|=c$ since $b=r+2$, a contradiction. In the second line we have $b=r+1$ and the equality follows. So, we may assume that $\min\{r-1,t-1\}=t-1<r-1$ and we compare $t+1<r+1$ with $b$. In first two lines we get a contradiction as well as in the third line of $b$. In the last line of $b$ we have the equality only when $t=r-1$ and otherwise a contradiction.

We continue with $Q_W\cup Q_X\cup Q_y\subset C$. At most $r$ vertices from $U\times(X\cup (Y-\{y_t\})\cup W)$ can be outside of $C$. This yields a contradiction with the first two lines of $b$. So let $r\geq q+3$. If $m_s=r-q$ in third line or $m_t=r-q$ in fourth line, then $b=r$ and the equality follows. thus, we may assume that $m_s=s<r-q$ in third line or $m_t=t-1<r-q$ in fourth line. If $s\geq t$, then we can have at most $q+s$ vertices from $U\times(X\cup Y\cup W)$ outside of $C$ by edges of types G, H, F and I, the equality follows. Otherwise, $s<t$ and we can have at most $q+t-1$ vertices from $U\times(X\cup Y\cup W)$ outside of $C$ by edges of types G, H, F and I, the equality follows again and the proof is completed.
\end{proof}


\section*{Acknowledgements}
The authors would like to thank the Slovenian Research and Innovation Agency for financing bilateral project between Slovenia and the USA (project No. BI-US/22-24-101).



\begin{thebibliography}{99}

\bibitem{BrDoKl} B. Bre\v{s}ar, P. Dorbec, W. Goddard, B. Hartnell, M.A.
Henning, S. Klav\v{z}ar, D.F. Rall, Vizing's conjecture: a survey and
recent results, J. Graph Theory 69 (2012) 46--76.


\bibitem{FeSc} J. Feigenbaum, A.A. Sch\"{a}ffer, Finding prime factors of
strong direct product in polynomial time, Discrete Math. 109 (1992) 77--102.

\bibitem{random_graph} A. Frieze, M. Karo\'{n}ski. Introduction to Random Graphs. Cambridge University Press, Cambridge, 2015.

\bibitem{HaIm} R. Hamack, W. Imrich, On Cartesian Skeletons of Graphs, Ars
Math. Contemp. 2 (2009) 191--205.

\bibitem{HaIK} R. Hamack, W. Imrich, S. Klav\v{z}ar, Handbook of Product
Graphs, Second Edition, CRC Press, Boca Raton, FL, 2011.

\bibitem{Imri1} W. Imrich, Assoziative Produkte von Graphen, Osterreich. Akad. Wiss. Math.-Natur. Kl. S.-B. H 180 (1972) 203--239. (In German.)

\bibitem{Imri} W. Imrich, Factoring cardinal product graphs in polynomial
time, Discrete Math. 192 (1998) 119--144.

\bibitem{ImIz} W. Imrich, H. Izbicki, Associative products of graphs, Monatsh. Math. 80 (1975) 277--281.

\bibitem{ImPe} W. Imrich, I. Peterin, Recognizing a Cartesian product in
linear time, Discrete Math. 307 (2007) 472--483.

\bibitem{KPY1} C.X. Kang, I. Peterin, E.Yi, On the simultaneous metric dimension of a graph and its complement, Rocky Mountain J. Math., to appear.

\bibitem{KPY2} C.X. Kang, I. Peterin, E.Yi, The simultaneous fractional dimension of graph families, Acta Math. Sin. (Engl. Ser.), 39 (2023) 1425--1441.

\bibitem{KY} C.X. Kang, E. Yi, The fractional strong metric dimension of graphs, Lecture Notes in Comput. Sci. 8287 (2013) 84--95.

\bibitem{KePe} A. Kelenc, I. Peterin, On some metric properties of direct-co-direct product, Appl. Math. Comput. 457 (2023) 128152.

\bibitem{KePe1} A. Kelenc, I. Peterin, Distance formula for direct-co-direct product in the case of disconnected factors, Art Discrete Appl. Math. 6 (2023) p2.13.

\bibitem{Kim} S.-R. Kim, Centers of a tensor composite graph, Congr.
Numer. 81 (1991) 193--203.

\bibitem{KPY} S. Klav\v{z}ar, I. Peterin, I.G. Yero, Graphs that
are simultaneously efficient open domination and efficient closed domination
graphs, Discrete Appl. Math. 217 (2017) 613--621.

\bibitem{strong-lex} D. Kuziak, I.G. Yero, J.A. Rodr\'iguez-Vel\'azquez, Closed formulae for the strong metric
dimension of lexicographic product graphs, Discuss. Math. Graph Theory. 36(4) (2016) 1051--1064.

\bibitem{Strong-strong} D. Kuziak, I.G. Yero, J.A. Rodr\'{\i}guez-Vel\'{a}zquez,
On the strong metric dimension of the strong products of graphs, Open Math. (formerly Cent. Eur. J. Math.) 13 (2015) 64--74.

\bibitem{Kuziak-Erratum} D. Kuziak, I.G. Yero, J.A. Rodr\'iguez-Vel\'azquez, Erratum to ``{O}n
  the strong metric dimension of the strong products of graphs'', Open Math. 13  (2015) 209--210.
	
\bibitem{KuPY} D. Kuziak, I. Peterin, I.G. Yero, Resolvability and Strong Resolvability in the Direct Product of Graphs, Results Math. 71 (2017) 509-526.

\bibitem{MaBK} P. Manuel, B. Bre\v{s}ar, S. Klav\v{z}ar, The geodesic-transversal problem, Appl. Math. Comput. 413 (2022) 126621.

\bibitem{Oellermann} O.R. Oellermann and J. Peters-Fransen, The strong metric dimension of graphs and digraphs, Discrete Appl. Math. 155 (2007) 356--364.

\bibitem{PeSe} I. Peterin, G. Semani\v{s}in, On maximal shortest paths cover, Mathematics 9(14) (2021) \#1592.

\bibitem{str-dim-cart-dir} J.A. Rodr\'iguez-Vel\'azquez, I.G. Yero, D. Kuziak, O.R. Oellermann,
On the strong metric dimension of Cartesian and direct products of graphs, Discrete Math. 335 (2014) 8--19.

\bibitem{seb} A. Seb\H{o}, E. Tannier, On metric generators of graphs, Math. Oper. Res. 29(2) (2004) 383--393.

\bibitem{shao-solis} Z.~Shao, R.~Solis-Oba, $L(2,1)$-labelings on the modular product of two graphs, Theoret. Comput. Sci. 487 (2013) 74--81.

\bibitem{Shit} Y. Shitov, Counterexamples to Hedetniemi’s conjecture, Annal. Math 190 (2019) 663--667.

\end{thebibliography}
\end{document}